\documentclass[reqno,11pt]{amsart}

\usepackage{amsmath}
\usepackage{mathrsfs}
\usepackage{amsfonts}
\usepackage{amsthm}
\usepackage{latexsym}
\usepackage[psamsfonts]{amssymb}
\usepackage[psamsfonts]{eucal}
\usepackage{graphicx}
\def\C{\mathcal C}
\def\constk{C}
\def\constl{c}

\def\gammaresc{\widetilde \gamma}
\def\gammaresctipoII{\widehat \gamma}
\def\gammaspec{\gamma^{\rm sp}}

\def\kappag{\kappa_\gamma}
\def\inidat{\overline\gamma}
\def\nuovoparametrotemporale{{\mathfrak t}}
\def\npt{\nuovoparametrotemporale}
\def\paraII{\mu}
\def\R{{\mathbb R}}
\def\regione{E}

\def\tangvers{\tau}

\numberwithin{equation}{section} 
\newtheorem{thm}{Theorem}[section]
\newtheorem{dfnz}[thm]{Definition}
\newtheorem{prop}[thm]{Proposition}

\newtheorem{rmk}[thm]{Remark}

\date{}

\begin{document}
\title[Curvature evolution of nonconvex lens-shaped domains]{Curvature evolution of nonconvex lens-shaped domains}

\author[Giovanni Bellettini]{Giovanni Bellettini}
\address[Giovanni Bellettini]{Dipartimento di Matematica, Univ. Roma 
Tor Vergata, via della Ricerca Scientifica,
00133 Roma, Italy, and INFN Laboratori Nazionali di Frascati,
via E. Fermi 40, Frascati (Roma),  
 Italy}
\email[G.~Bellettini]{Giovanni.Bellettini@lnf.infn.it}

\author[Matteo Novaga]{Matteo Novaga}
\address[Matteo Novaga]{Dipartimento di Matematica, Univ. Padova,  
via Trieste 63, 35121 Padova, Italy}
\email[M.~Novaga]{novaga@math.unipd.it}

\keywords{Curvature flow, blow-up singularities,
triple junctions}
\subjclass{
Primary 53C44; Secondary 35B40, 53A04.}

\begin{abstract}
We study the curvature flow  of planar 
nonconvex lens-shaped domains, 
considered as special symmetric networks with two triple junctions.
We show that 
the evolving domain becomes convex in finite time; then it shrinks 
homothetically to a point, as proved in \cite{SS:08}.
Our theorem is the 
analog of the result of Grayson \cite{Gr:87} for 
curvature flow of closed planar embedded curves.
\end{abstract}

\maketitle

\section{Introduction}\label{sec:intro}
Mean curvature flow of partitions, in particular of planar networks,
has been considered by various authors, see for instance 
\cite{Mu:56}, \cite{Br:78},
 \cite{BR:93}, \cite{De:96}, \cite{MNT:04}, 
\cite{Fr:08}, \cite{Sa:09}. 
Such a geometric flow
is a generalization of mean curvature flow, when more than two
phases are present. The main difficulties are
due to the presence of multiple junctions, typically
triple points in the planar case. 

In this paper we consider the curvature flow of a lens-shaped 
network, that is, of a
particular planar network symmetric with respect to the first
coordinate axis, 
and having there
two triple junctions.
If the bounded region enclosed by the network is convex, it is proved in 
\cite{SS:08} that the evolution remains convex and shrinks to a point in 
finite time, 
while its shape approaches
a unique profile $\gamma^{{\rm h}}$,
 corresponding to a homothetically shrinking
solution (see \cite[Fig. 1]{SS:08}).
This is the precise analog of the well-known
result of Gage and Hamilton 
\cite{GH:86},
which shows that a closed convex planar curve evolving by curvature shrinks 
to a point in finite time, approaching a circle. This result has been generalized
by Grayson \cite{Gr:87} 
who showed  that a 
closed nonconvex initial embedded curve 
has no singularities before the extinction, it 
becomes convex and eventually shrinks to a point.
A different proof of Grayson's theorem was given by Huisken
in \cite{Hu:98}.

Our aim is to study the long time 
curvature evolution of a general (not
necessarily convex) lens-shaped network. 
We will show that 
such a network 
becomes convex in finite time and eventually shrinks 
homothetically to a point,
as described in \cite{SS:08}.
Our result is, therefore,
the analog of the result of Grayson, but  
in the context of curvature flow of networks.
Our  proof is based on the classification of all possible singularities, 
in analogy to the proof given in \cite{Hu:98} for curvature flow of curves.
We point out that
in the evolution considered here we are able
to overcome the technical difficulties 
which prevented in \cite{MNT:04}
the complete analysis of type II singularities.

The  main result of the present paper, which 
is a consequence of Theorems \ref{th:exist}, 
\ref{th:singuno} and \ref{teosingII}, reads as follows.
\begin{thm}\label{thth}
Assume that the initial curve
 $\inidat: [0,1]\to \R^2$ satisfies the 
regularity and compatibility conditions listed in 
assumption (A) 
(Section \ref{sub:defgammaetc}) and 
is embedded (hypothesis
\eqref{asspos}). Then there exist
$T \in (0,+\infty)$ and a solution $\gamma\in \C^{2,1}([0,1]\times [0,T))$ 
of the evolution problem \eqref{eqevol} expressing the curvature
flow of a symmetryc network with two triple junctions, 
such that 
\begin{eqnarray*}
L(\gamma(t)) &\le& \constk \sqrt{2(T-t)}, \qquad t \in [0,T),
\\
\|\kappa_{\gamma(t)}\|_{L^\infty([0,1])}
 &\le& \frac{\constk}{\sqrt{2(T-t)}},
\qquad t \in [0,T),
\end{eqnarray*}
where $L(\gamma(t))$ and 
$\kappa_{\gamma(t)}$ 
denote the length and the curvature of $\gamma(t)$ respectively, 
and $C$ 
is an absolute positive constant.
Moreover, there exists $\overline t\in [0,T)$ such that 
the region $\regione(\gamma(t))$ enclosed by the 
corresponding network is
uniformly convex  for all $t\in [\overline t, T)$, and 
$T$ is the extinction time of the evolution, i.e. 
$$
\lim_{t \to T^-} L(\gamma(t)) = \lim_{t \to T^-}
\vert \regione(\gamma(t))\vert =0.
$$
Finally,   a suitable
rescaled and translated version of 
$\gamma(t)$  
converges in $\C^2([0,1];\R^2)$ to $\gamma^{{\rm h}}$  as $t\to T^-$.
\end{thm}

We note that to prove Theorem \ref{thth}
the only result needed from
\cite{SS:08} is 
 the uniqueness of $\gamma^{{\rm h}}$. 

In the last section  of the paper we exhibit
two examples of singularities appearing
 {\it before} the extinction time.
In Example 1 we show the formation of a singularity, 
starting from a suitable immersed initial datum $\inidat$
(see Fig. \ref{exa_typeII}); in this case
the $L^\infty$-norm of the curvature of $\gamma(t)$ blows up at $t=T$, and 
$T$ is smaller than the extinction time. 
In Example 2,
starting from an embedded double-bubble shaped $\inidat$
as in  Fig. \ref{centoventi}
(hence with different Neumann boundary conditions
with respect to the ones in Theorem \ref{thth})
 we show that the singularity appears at $t=T$ before the extinction time,
due to the collision of the two triple junctions.

We conclude this introduction by mentioning that a general
analysis of curvature flow of planar networks has been recently
announced to the second author by Tom Ilmanen \cite{IS}. 

\section{Notation}\label{sec:not}
Given 
$T>0$ and a map
$\gamma = (\gamma_1,\gamma_2) : [0,1] \times [0,T) \to \R^2$, 
 for
$t \in [0,T)$ 
we set
$\gamma(t) : [0,1] \to \R^2$, $\gamma(t)(x) := \gamma(x,t)$. If
$\gamma \in \C^{2,1}([0,1] \times [0,T);\R^2)$, we introduce
the following notation:
\begin{itemize}
\item[-] 
$L(\gamma(t)) := \int_0^1 |\gamma_x(x,t)|\, dx$ is 
the length of $\gamma(t)$, where $\gamma_x$ denotes the derivative 
with respect to $x$;
\item[-]  $s \in I(t):= [0, L(\gamma(t))]$ is the 
(time dependent) arclength parameter of $\gamma(t)$, and 
 $\partial_s := \dfrac{\partial_x}{|\gamma_x|}$ 
denotes the derivative with respect to $s$;
\item[-] 
$\tangvers_{\gamma(t)} = \tangvers(t)  = (\tangvers_1(t), \tangvers_2(t)) 
:= \gamma_s(t)$ is the unit tangent vector to $\gamma(t)$,
and  $\tangvers(t)(x) := 
\tangvers(x,t)$;
\item[-] $\nu_{\gamma(t)}= \nu(t) := \left(-\tangvers_2(t),\tangvers_1(t)
\right)$ is the normal vector to $\gamma(t)$
obtained by rotating $\tangvers(t)$ 
counterclockwise 
of $\pi/2$, and $\nu(t)(x) := \nu(x,t)$;
\item[-]
$\kappa_{\gamma(t)} := 
\langle \tangvers_s(t) ,\nu(t)\rangle = 
\langle \frac{\gamma_{xx}(t)}{\vert \gamma_x(t)\vert^2}, \nu(t)\rangle$ 
is the curvature of $\gamma(t)$, and
$\kappag(x,t):= \kappa_{\gamma(t)}(x)$;
\item[-] $\gamma_t :=\partial_t \gamma$ denotes the derivative
of $\gamma$ with
respect to $t$.
\end{itemize}
We denote by $\vert E\vert$ the Lebesgue measure of a measurable set
$E \subseteq \R^2$.
\subsection{The geometric evolution equation}\label{sub:geom}
We are concerned with the following geometric evolution problem:
\begin{equation}\label{eqevol}\left\{
\begin{array}{ll}
\displaystyle 
\gamma_t &=\, \dfrac{\gamma_{xx}}{|\gamma_x|^2} \qquad \qquad 
\quad {\rm in}~ (0,1) \times 
(0,T),
\\
\gamma_2(0,t)&=\,\gamma_2(1,t)\,=\,0 \qquad t \in (0,T),
\\
\tangvers(0,t) &= \,\Big( \dfrac 1 2,\dfrac{\sqrt 3}{2}\Big)
\quad \  \qquad t \in (0,T),
\\
\tangvers(1,t) &= \,\Big(\dfrac 1 2,-\dfrac{\sqrt 3}{2}\Big)
\qquad \ \ t \in (0,T),
\\
\gamma(0) &= \,\inidat 
\qquad
\qquad
\ \ 
\qquad {\rm in}~ (0,1)
\end{array}\right.
\end{equation}
where the initial curve $\inidat = (\inidat_1, \inidat_2)
\in \C^2([0,1]; \R^2)$  satisfies 
\begin{equation}\label{doremi}
\vert \inidat_x(x)\vert 
\neq 0, \qquad x \in [0,1],
\end{equation}
 and the compatibility conditions
\begin{equation}\label{ipo:datum}
\inidat_2(0) = \inidat_2(1)=0, \qquad
\frac{\inidat_x(0)}{\vert \inidat_x(0)\vert} = 
\Big(
 \dfrac 1 2,\dfrac{\sqrt 3}{2}\Big), \qquad 
\frac{\inidat_x(1)}{\vert \inidat_x(1)\vert} = \Big(
\dfrac 1 2,- \dfrac{\sqrt 3}{2}\Big).
\end{equation}

System \eqref{eqevol} corresponds to motion by curvature 
(first equation) of a planar curve 
with the extremal
points $\gamma(0,t)$, $\gamma(1,t)$ sliding on the first coordinate axis
(second equation), and satisfying
the following Neumann boundary conditions (third and fourth 
equation):
\begin{equation}\label{neumangle}
{\rm angle~ between~} 
e_1 {\rm ~and~} 
\tangvers(t) =
\begin{cases}
& \pi/3 {\rm ~at}~ \gamma(0,t) = (\gamma_1(0,t),0),
\\
& - 
\pi/3 {\rm~ at~} \gamma(1,t) = (\gamma_1(1,t),0),
\end{cases}
\end{equation}
where $e_1 := (1,0)$.

\subsection{Definitions of  
$\gammaspec$
and $\lambda$}\label{sub:defgammaetc}
For $t \in [0,T)$ 
we define the ``specular'' curve
$\gammaspec := (\gamma_1,-\gamma_2)$.
The corresponding network mentioned in the Introduction  is the one 
formed by $\gamma([0,1],t) \cup \gammaspec([0,1],t)$ 
and by the two horizontal half lines $(-\infty, \gamma_1(0,t))$
and $(\gamma_1(1,t), +\infty)$ lying
on the first coordinate axis. 

In the following, we let the function $\lambda_\gamma = 
\lambda: [0,1]\times [0,T)\to \R$ be such that 
\begin{equation}\label{eq:lambdakappa}
\gamma_t = \lambda\,\tangvers + \kappa\,\nu.
\end{equation}
Note that
\begin{equation}\label{mifasol}
\lambda = \langle \gamma_t, \tangvers\rangle
= \langle \frac{\gamma_{xx}}{\vert\gamma_x\vert^2}, \tangvers\rangle.
\end{equation}
Formally
differentiating in time the boundary conditions in \eqref{eqevol} 
(second equation)
and using
\eqref{eq:lambdakappa}
we have at $(0,t)$ and $(1,t)$ the relation
$0=\partial_t \gamma_{2} = \lambda \tangvers_2 + \kappag \nu_2$, which gives
\begin{equation}\label{eqkappalambdadiffuno}
\kappag (0,t) = -\sqrt 3\lambda(0,t), \qquad
\kappag (1,t) = \sqrt 3\lambda(1,t),
\end{equation}
where we make use of the third and fourth equations in \eqref{eqevol}.
Moreover, recalling from \cite[formula (2.4)]{MNT:04} that
$\tangvers_t = (\partial_s \kappag + \lambda\kappag)\nu$, 
we find 
\begin{equation}\label{eqkappasdiffdue}
\partial_s \kappag(0,t) +\lambda(0,t) \kappag(0,t)  = 
\partial_s \kappag(1,t) +\lambda(1,t) \kappag(1,t)  = 0.
\end{equation}
Notice that 
\eqref{eqkappasdiffdue} and
\eqref{eqkappalambdadiffuno} imply
\begin{equation}\label{behh}
\begin{aligned}
\partial_s \kappag(0,t) &= - \lambda(0,t) \kappag(0,t) = \frac{\kappag(0,t)^2}{\sqrt{3}}
\geq 0,
\\
\partial_s \kappag(1,t) &= - \lambda(1,t) \kappag(1,t) = -\frac{\kappag(1,t)^2}{\sqrt{3}}
\leq 0
\end{aligned}
\end{equation}
for all 
$t\in (0,T)$.
In particular, 
the function 
$\kappa_{\gamma(t)}$ can never attain its maximum at $x=0$
unless 
$\kappag(0,t) = \partial_s \kappag(0,t)=0$;  similarly
$\kappa_{\gamma(t)}$ can never attain its maximum at $x=1$
unless 
$\kappag(1,t) = \partial_s\kappag(1,t)=0$. 

\smallskip

From now on we will always make the following 
assumption (A) on $\inidat$:
\begin{itemize}
\item[(A)]
$\inidat \in \C^2([0,1];\R^2)$  
 satisfies \eqref{doremi}, \eqref{ipo:datum} 
and the second
order compatibility conditions 
\begin{equation}
\langle
\inidat_{xx}(0), \overline \nu(0)\rangle = -\sqrt{3}\langle
\inidat_{xx}(0), \overline \tangvers(0)\rangle, \qquad 
\langle \inidat_{xx}(1), \overline \nu(1)\rangle = -\sqrt{3}\langle
\inidat_{xx}(1), \overline \tangvers(1)\rangle, \qquad 
\end{equation} 
where $\overline \tangvers= \frac{\inidat_x}{\vert \inidat_x\vert}=
(\tangvers_1, \tangvers_2)$ and
$\overline \nu := (-\overline \tangvers_2, \overline \tangvers_1)$.
\end{itemize}

\smallskip

Note that under the sole assumption (A) the set $\overline 
\gamma([0,1],t)$ may have
self-intersections, see Fig. \ref{fig:immersed}. 

\begin{figure}
\begin{center}
\includegraphics[height=3.0cm]{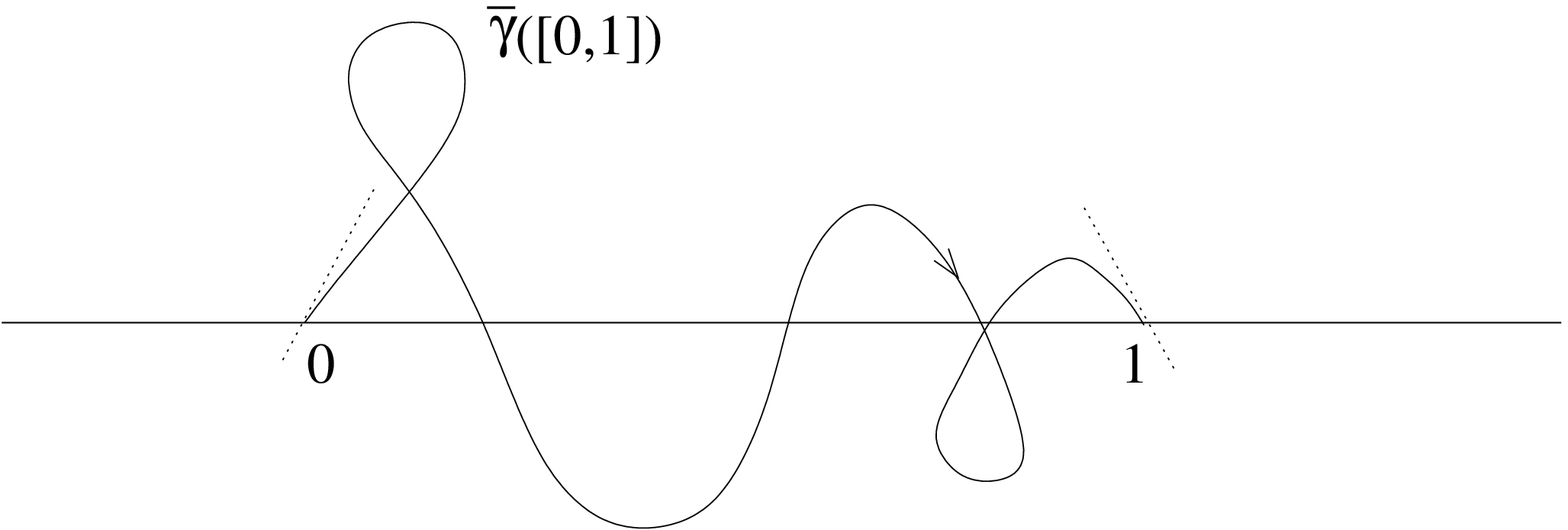}
\caption{{\small An immersed initial datum $\inidat$ satisfying 
assumption (A).}}
\label{fig:immersed}
\end{center}
\end{figure}

\begin{dfnz}
We will refer to the embedded case, provided 
\begin{equation}\label{asspos}
\inidat {\it ~is~ injective~
and~} \inidat_2(x)> 0 {\it ~for~ all~} x\in (0,1). 
\end{equation}
\end{dfnz}
In the embedded case $\inidat([0,1])$ is not necessarily
a graph with respect to the first coordinate axis. However, we can speak
of the connected bounded plane region $\regione(\inidat)$ in between 
$\inidat([0,1])$ and 
$\inidat^{{\rm sp}}([0,1])$,  see Fig. \ref{fig:embedded}. 

\begin{figure}
\begin{center}
\includegraphics[height=4.0cm]{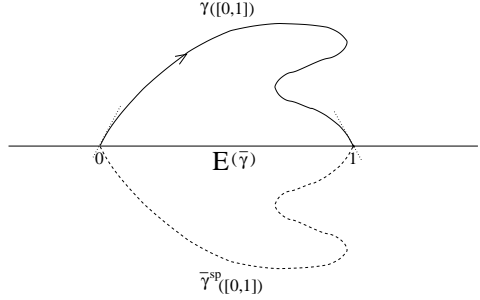}
\caption{{\small An embedded initial datum $\inidat$, with its specular one
(dotted curve)
and the region $\regione(\inidat)$ enclosed between the two curves. The points
$(0,0)$ and $(1,0)$ are the two triple junctions, if
one imagines to add to the curves the horizontal half lines on the 
left of $(0,0)$ and on the right of $(1,0)$.}}
\label{fig:embedded}
\end{center}
\end{figure}

We will refer to the {\it convex} case, 
provided 
$$
\inidat((0,1)) {\rm ~is~ the~ graph~ of~ a~ positive~ concave~ function}.
$$
The convex case is in particular embedded,
and has been studied in \cite{SS:08},
where it is proven that 
$\gamma(t)$ remains concave.
Therefore, the plane region $\regione(\gamma(t))$ between $\gamma([0,1],t)$ and 
$\gammaspec([0,1],t)$ is still well defined, it 
is 
 a convex lens-shaped domain 
evolving by curvature, 
and having the
two singular points $\gamma(0,t)$, $\gamma(1,t)$ in its
boundary. 
 
\begin{rmk}\label{rem:conve}
With our convention, 
 in the convex case
$\kappa_{\gamma(t)}$ is negative, since $\gamma(t)$ 
is parametrized in such a way that $\regione(\gamma(t))$
lies locally on the right of $\gamma(t)$.
\end{rmk}

\subsection{The homotetically shrinking solution
$\gamma^{{\rm h}}$.}\label{rem:homotetico}
In  \cite{CG:07},
\cite{SS:08}
it is proven that 
there exists a unique embedding 
 $\gamma^{{\rm h}}
\in \C^\infty([0,1]; \R^2)$ which satisfies 
$\gamma^{{\rm h}}_2(0) = \gamma^{{\rm h}}_2(1) = 0$,
$
\frac{
\gamma^{{\rm h}}_x(0)}{\vert \gamma^{{\rm h}}_x(0)
\vert}
 = (\frac{1}{2}, \frac{\sqrt{3}}{2})$,
$
\frac{\gamma^{{\rm h}}_x(1)}{\vert \gamma^{{\rm h}}_x(1)\vert}
 = (\frac{1}{2}, -\frac{\sqrt{3}}{2})$, 
which gives raise to a homothetically shrinking
curvature evolution, namely 
\begin{equation}\label{eq:homote}
\kappa_{\gamma^{{\rm h}}} + \langle \gamma^{{\rm h}}, \nu_{\gamma^{{\rm h}}} \rangle =0
\qquad {\rm in}~ (0,1).
\end{equation}
Moreover
$$
\inf_{x \in (0,1)}\kappa_{\gamma^{{\rm h}}}(x) >0.
$$
%
\section{Immersed initial data}\label{sec:imm}
In the next theorem $\gamma([0,1],t)$ is allowed to have self-intersections.

\begin{thm}\label{th:exist}
Assume that $\inidat$ satisfies (A).
Then problem
\eqref{eqevol} has a unique solution
$$
\gamma \in \C^\infty([0,1] \times (0,T);\R^2) 
\cap \C^{2,1}([0,1]\times [0,T);\R^2),
$$
 defined on a maximal time interval $[0,T)$, and $T<+\infty$. 
Moreover
\begin{equation}\label{ksing}
\limsup_{t\to T^-} \|\kappa_{\gamma(t)}\|_{L^2([0,1])} = 
+\infty.
\end{equation}
\end{thm}
\begin{proof}
All assertions but $T < +\infty$ follow from \cite[Theorems 3.1, 3.18 and 
Remark 3.24]{MNT:04}. Let us show that $T < +\infty$.
Take an initial open convex bounded lens-shaped 
domain $\regione(\overline \eta)$ 
with 
$$
\regione(\overline \eta) \supset 
\inidat([0,1]),
$$
whose boundary is given by $\overline \eta([0,1]) \cup 
\overline \eta^{\rm sp}([0,1])$,
where $\overline \eta: [0,1]\to \R^2$ gives raise to a homothetically
shrinking curvature evolution $\eta: 
[0,1]\times [0,t^*)\to \R^2$,  
$t^*<+\infty$, 
with the same boundary conditions as  $\gamma$, i.e., 
\begin{equation}\label{bdryeta}
\eta_2(0,t) = \eta_2(1,t) = 0, \quad
\frac{\eta_x(0,t)}{\vert \eta_x(0,t)\vert} = 
\left(\frac{1}{2}, \frac{\sqrt{3}}{2}\right), \quad
\frac{\eta_x(1,t)}{\vert \eta_x(1,t)\vert} = \left(
\frac{1}{2}, -\frac{\sqrt{3}}{2}\right), 
\end{equation}
see Fig. \ref{fig:comparison} and Section \ref{rem:homotetico}. 

We claim that the following comparison principle
holds: 
\begin{equation}\label{eq:compa}
\regione(\eta(t))  \supset \gamma([0,1],t),
\end{equation}
for all times $t \in [0,t^\#)$, where $t^\#:= \min(t^*,T)$.

Since the proof of this comparison result differs slightly
from the standard comparison proof for curvature flow,
we indicate here the main steps. 
Define 
$$
\delta(t) := {\rm dist}\left(
\partial \regione(\eta(t)), \gamma([0,1],t)
\right), \qquad
t \in [0,t^\#).
$$
To prove \eqref{eq:compa}, it is enough
to show that 
\begin{equation}\label{eq:deltaprimepositive}
\lim_{h \to 0^+}
\frac{\delta(t+h)-\delta(t)}{h} \geq 0, \qquad t \in (0,t^\#).
\end{equation}
For any $(x,\xi,t) \in [0,1]^2\times [0,T)$ 
set
$$
u(x, \xi,t):= \vert \eta(x,t) - \gamma(\xi,t)\vert, 
\qquad 
v(x, \xi,t):= \vert \eta^{{\rm sp}}(x,t) - \gamma(\xi,t)\vert.
$$
It is well known (see for instance \cite{Ha:86}) that
$$
\lim_{h \to 0^+}
\frac{\delta(t+h)-\delta(t)}{h} =
\min 
\Big( U(t), V(t) \Big)
$$
where 
$$
\begin{aligned}
U(t) & := \min
\Big\{
\frac{\partial u}{\partial t}(x,\xi,t) : 
(x, \xi) \in [0,1] \times [0,1], 
\delta(t) = u(x,\xi,t) \Big\},
\\
V(t) & := 
\min 
\Big\{ \frac{\partial v}{\partial t}(y,\eta,t)
: 
(y, \eta) \in [0,1] \times [0,1], 
\delta(t) = v(y,\eta,t) \Big\}.
\end{aligned}
$$
Given $t \in (0,t^\#)$, 
we denote by $x^t,\xi^t \in [0,1]$
two parameters for which either 
$\delta(t) = u(x^t,\xi^t,t)$, or 
$\delta(t) = v(x^t,\xi^t,t)$. Without loss of generality, we assume
$\delta(t) = u(x^t,\xi^t,t)$, and we set
$$
q^t := \eta(x^t,t),
\qquad p^t := \gamma(\xi^t,t),
$$
see Fig. \ref{fig:comparison}.
Note that
\begin{equation}\label{qeta}
q^t \notin \{\eta(0,t), \eta(1,t)\}. 
\end{equation}
Indeed  if by contradiction we have for instance 
$q^t = \eta(0,t)$ then,
 in view of the Neumann boundary conditions in \eqref{eqevol} and \eqref{bdryeta}, 
the distance between $p^t$ and  a point
$q$ on $\partial \regione(\eta(t))$ would decrease when $q$ moves from 
$q^t$ sliding slightly either on $\eta([0,1],t)$ or on 
$\eta^{{\rm sp}}([0,1],t)$.

\begin{figure}
\begin{center}
\includegraphics[height=3.5cm]{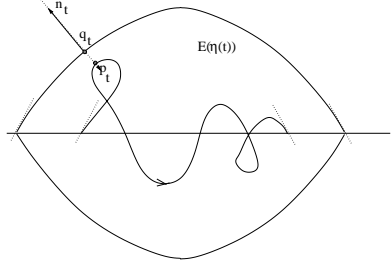}
\caption{{\small The inner curve is $\gamma(t)$, the outer
curve is $\eta(t) \cup \eta^{{\rm sp}}(t)$, bounding the 
self-similar shrinking convex set $\regione(\eta(t))$}.}
\label{fig:comparison}
\end{center}
\end{figure}

We now distinguish two cases.

{\it Case 1}. $p^t \notin \{\gamma(0,t), \gamma(1,t)\}$, see Fig. \ref{fig:comparison}.
In this case,  thanks to \eqref{qeta},
we are reduced
to the standard situation of curvature flow
(see for instance \cite{An:90}), and \eqref{eq:deltaprimepositive}
follows. 

{\it Case 2}. 
$p^t \in \{\gamma(0,t), \gamma(1,t)\}$. Without loss of generality,
we can assume that $p^t = \gamma(1,t)$, and that the second 
component of $q^t$ is positive. 
 Let $n^t:= \frac{q^t-p^t}{\vert
q^t-p^t\vert}$. Then it is not difficult to 
see that $n^t$ equals the  unit normal to $\partial \regione(\eta(t))$
at $q^t$ pointing out of $\regione(\eta(t))$. Let
$K := \{(\cos\theta, \sin\theta) : \theta \in [0,\pi/6]\}$.
If $n^t \in \partial K$ 
 then
again \eqref{eq:deltaprimepositive} follows in a standard way. 
On the other hand, we cannot
have   $n^t = (\cos \theta^t, \sin\theta^t) : \theta^t \in 
[0,\pi/6)\}$, since this contradicts  
the Neumann boundary conditions in \eqref{bdryeta} and the convexity 
of $\eta(t)$. 

The proof of \eqref{eq:compa} is concluded, and in particular
$T\le t^*$.
\end{proof}

Note that the smoothness of $\gamma$ implies that 
 $\Vert \kappa_{\gamma(t)}\Vert_{L^\infty([0,1])}$ is finite for all
$t \in [0,T)$. On the other hand, from \eqref{ksing} we deduce 
that 
\begin{equation}\label{linftyscoppia}
\limsup_{t\to T^-} \|\kappa_{\gamma(t)}\|_{L^\infty([0,1])} = 
+\infty.
\end{equation}

\begin{prop}\label{rem:lunghezza}\rm
There exists a constant $\constl 
>0$ independent of $\gamma$ such that 
\begin{equation}\label{LC}
L(\gamma(t)) \leq \constl 
L(\inidat), \qquad t \in [0,T).
\end{equation}
\end{prop}
\begin{proof}
Since $\gamma_t(0,t)$ and $\gamma_t(1,t)$ are horizontal,
it follows from \eqref{mifasol} that $\lambda(0,t) = \partial_t
\gamma_1(0,t)/2$, 
and $\lambda(1,t) = \partial_t \gamma_1(1,t)/2$.
Observing (see \cite[Proposition 3.2]{MNT:04})
that the time-derivative of the measure $ds$ is given by
\begin{equation}\label{eq:changelungh}
(\lambda_s - \kappag^2)\,ds,
\end{equation}
we have
\begin{eqnarray}\label{jelka}
\frac{d}{dt} L(\gamma(t)) &=& 
 \lambda\big\vert^{x=1}_{x=0}
-\int_{I(t)} \kappa_{\gamma(t)}^2\,ds 
=\frac{1}{2} \left( \partial_t \gamma_1(1,t) - 
\partial_t \gamma_1(0,t)\right) 
-\int_{I(t)} \kappa_{\gamma(t)}^2\,ds 
\\
\nonumber 
&\leq& 
\frac{1}{2} \left( \partial_t \gamma_1(1,t) - 
\partial_t \gamma_1(0,t)\right).
\end{eqnarray}
Hence
\begin{equation}\label{stimaelle}
L(\gamma(t)) \le L(\inidat) - \frac{1}{2} \big( 
\gamma_1(1,0) - \gamma_1(0,0)\big)
+ 
\frac{1}{2} \left( 
\gamma_1(1,t) - \gamma_1(0,t)
\right).
\end{equation}
Therefore, to conclude the proof it is enough to show that 
$\gamma_1(1,t) - \gamma_1(0,t)$
 is bounded by 
$\constl L(\inidat)$, where $\constl>0$ is an absolute constant
independent of $\gamma$. This  assertion can be proved 
by a comparison argument as in the proof of Theorem \ref{th:exist}:
taking a lens-shaped convex domain as in Theorem \ref{th:exist}, 
it follows that the horizontal length 
$\gamma_1(1,t) - \gamma_1(0,t)$ cannot be larger 
than the corresponding horizontal length of $\regione(\eta(t))$, which 
can be bounded by an absolute constant times $L(\inidat)$.
\end{proof}

Following \cite{Hu:90} and recalling \eqref{linftyscoppia}, we say that: 
\begin{itemize}
\item[$\bullet$] $\gamma$ develops a {\it type I singularity} at $t=T$
if there exists $\constk > 0$ such that 
\begin{equation}\label{eqsI}
\|\kappa_{\gamma(t)}\|_{L^\infty([0,1])}\le \frac{\constk}{\sqrt{2(T-t)}},
\qquad t \in [0,T).
\end{equation}
\item[$\bullet$]
$\gamma$ develops a {\it type II singularity} at $t =
T$ if
$$
\limsup_{t \to T^-} \sqrt{2(T-t)} \,\Vert 
\kappa_{\gamma(t)}\Vert_{L^\infty([0,1])} = +\infty.
$$
\end{itemize}
Before passing to the next result, 
we recall from \cite[Eq. (2.6)]{MNT:04} that 
 the evolution equation for $\kappa$ reads as follows:
\begin{equation}\label{eqcurv}
\partial_t \kappag \,=\, \partial_{ss} \kappag + \lambda
\partial_s \kappag + \kappag^3.
\end{equation}
Note that this equation, being local, is valid
under the sole assumption (A).

The next observation 
is used to prove Proposition \ref{prop:tripodi}, which in turn will be used 
to prove Theorem \ref{teosingII}.

\begin{rmk}\label{rem:anal}
The solution $\gamma$ of \eqref{eqevol}
is analytic in $(0,1) \times
(0,T)$; in particular, for a given $t \in (0,T)$, the set 
$$
z(t) := \left\{x
\in [0,1]: \kappa_{\gamma(t)}(x)=0\right\}
$$
is finite. 
\end{rmk}

\begin{prop}\label{prop:tripodi} For any $t \in [0,T)$ we have 
\begin{equation}\label{eq:absvalkappa}
\frac{d}{dt} \int_{I(t)}|\kappa_{\gamma(t)}|\,ds = 
- 2\sum_{x \in z(t)} |\partial_s \kappag(x,t)|\ \le\ 0\,.
\end{equation}
\end{prop}
\begin{proof}
Using  Remark \ref{rem:anal},  \eqref{eq:changelungh} and
\eqref{eqcurv}
we compute 
\begin{equation}\label{eqstoma}
\begin{aligned}
\frac{d}{dt} \int_{I(t)}|\kappa_{\gamma(t)}|\,ds =& 
\int_{I(t)} \left[
\frac{\kappag}{|\kappag|} \partial_t\kappag +
(\lambda_s- \kappag^2)\vert \kappag\vert\right]
\,ds
\\
=& 
\int_{I(t)}
\left[ \frac{\kappag}{|\kappag|} \partial_{ss} \kappag + 
(\lambda|\kappag|)_s\right]\,ds.
\end{aligned}
\end{equation}
Integrating by parts we have
\begin{eqnarray}\label{eqpart}
\int_{I(t)} 
\frac{\kappag}{|\kappag|} \partial_{ss}\kappag \,ds = 
\frac{\kappag}{|\kappag|}  \partial_s\kappag 
\big\vert^{x=1}_{x=0} - 
\int_{I(t)} 
\left( 
\frac{\kappag}{\vert \kappag\vert}
\right)_s \partial_s\kappag ~ds.
\end{eqnarray}
Moreover
\begin{equation}\label{tresedici}
\int_{I(t)} 
\left( 
\frac{\kappag}{\vert \kappag\vert}
\right)_s \partial_s\kappag ~ds = 2 
\sum_{x \in z(t)}|\partial_s\kappag(x,t)|.
\end{equation}
Hence from \eqref{eqstoma}, \eqref{eqpart} and \eqref{tresedici} we deduce
\begin{equation}\label{eqstomabuis}
\begin{aligned}
\frac{d}{dt} \int_{I(t)}|\kappa_{\gamma(t)}|\,ds
=& - 2\sum_{x \in z(t)}|\partial_s \kappag(x,t)| + 
\frac{\kappag}{|\kappag|}
\left( \partial_s\kappag 
+\lambda\kappag\right)
\big\vert^{x=1}_{x=0}
\\
=& 
- 2
\sum_{x \in z(t)}|\partial_s \kappag(x,t)|
\ \le\ 0\,.
\end{aligned}
\end{equation}
\end{proof}
\section{Embedded nonconvex initial data: type I singularities}\label{secpos}
In this section, as well
as in Section \ref{secpostwo}, we  consider the 
embedded case. We begin to show
that
embeddedness is a property which is preserved by the evolution.

\begin{prop}\label{prop:embedded}
Assume that $\inidat$ satisfies (A) and 
\eqref{asspos}. Then 
\begin{itemize}
\item[(i)] for any $t \in [0,T)$ 
\begin{equation}\label{asspost}
\gamma(t) {\it ~is~ injective~
and~} \gamma_2(x,t)> 0 {\it ~for~ all~} x\in (0,1);
\end{equation}
\item[(ii)] for any $t \in [0,T)$ 
\begin{equation}\label{eqvol}
\vert \regione(\gamma(t))\vert = - \frac{4 \pi}{3} t
+ \vert \regione(\inidat)\vert.
\end{equation}
\end{itemize}
\end{prop}
\begin{proof}
Let $\delta := \sup \{t\in [0,T) : \gamma(t) {\rm ~is~ injective~ for~}
t\in [0,\delta)\}$. By \eqref{asspos}
and the smoothness of the evolution 
it follows that $\delta >0$. 
Given $(x,y,t)\in [0,1]^2\times [0,\delta)$ with $x< y$, 
let $S(x,y,t)$ be the relatively open segment 
 connecting
$\gamma(x,t)$ with $\gamma(y,t)$. Provided
$S(x,y,t)\cap \gamma([x,y], t)) = \emptyset$,
 we let $A^\gamma(x,y,t)$ be the subset of $\R^2$ 
bounded by  $\gamma([x,y],t)$ 
and $S(x,y,t)$.

Given $x,y \in [0,1]$ and $t \in [0,\delta)$, 
let also $\Sigma(x,y,t)$ be the relatively open segment 
 connecting
$\gamma(x,t)$ with $\gammaspec(y,t)$. Provided
$\Sigma(x,y,t)\cap  \partial \regione(\gamma(t))=\emptyset$, we have that 
either 
$\regione(\gamma(t)) \setminus \Sigma(x,y,t)$ is the union 
of two connected regions, or
$\big(\R^2 \setminus \regione(\gamma(t))\big) \setminus \Sigma(x,y,t)$ is the union 
of two connected regions.
We denote by 
$A_{\min}^\gamma(x,y,t)$ the region of minimal area
among  these two regions.

We define the function $g_\gamma:[0,\delta) \to [0,+\infty)$ as 
follows: for $t\in [0,\delta)$, 
\begin{equation}\label{def:g}
g_\gamma(t) := \min\big( Q_1^\gamma(t), Q_2^\gamma(t) \big),
\end{equation}
where 
\begin{eqnarray}
\label{defQ1}
Q_1^\gamma(t) & := & 
\inf_{x,y\in [0,1],\,x<y, S(x,y,t)\cap \gamma([x,y], t) = \emptyset} 
\frac{\left|\gamma(x,t)-\gamma(y,t)\right|^2}{\vert A^\gamma(x,y,t)\vert},
\\
\label{defQ2}
Q_2^\gamma(t) & := &
\inf_{x,y\in [0,1], \Sigma(x,y,t)\cap \partial \regione(\gamma(t))=
\emptyset} 
\frac{\left|\gamma^{{\rm sp}}(x,t)-\gamma(y,t)
\right|^2}{\vert A_{\min}^\gamma(x,y,t)\vert}.
\end{eqnarray}
Note that $g_\gamma$ is invariant under rescalings of 
$\gamma$, i.e., 
\begin{equation}\label{eq:invresc}
\vartheta > 0 \Rightarrow
g_{\vartheta \gamma}(t) = g_{\gamma}(t), \qquad t \in [0,\delta).
\end{equation}
By assumption \eqref{asspos}
it follows that 
\begin{equation}\label{gposo}
g_\gamma(0) >0.
\end{equation}
{}From \cite[Prop. 4.4]{MNT:04} it follows that 
$g_\gamma$ is increasing 
in every time subinterval of $[0,\delta)$ 
where it is strictly less than $4\sqrt{3}$.
In particular \eqref{gposo} implies
\begin{equation}\label{eq:gmon}
 g_\gamma(t)\ge \min\left(g_\gamma(0), 4 \sqrt{3}\right), \qquad t\in [0,\delta).
\end{equation}
{}From \eqref{eq:gmon} 
it follows that $\delta = T$, and (i) is proved.

Finally 
\begin{equation}\label{met}
\frac{1}{2} 
\frac{d}{dt} 
\vert \regione(\gamma(t))\vert = \int_{I(t)} \kappa_{\gamma(t)}\,ds = -\frac 2 3 \pi,
\end{equation}
which gives \eqref{eqvol}.
\end{proof}

\subsection{Type I singularities}
As usual in the blow-up analysis of type I singularities,
let us define the parameter $\npt$  as 
\[
 t(\npt):= T-e^{-2\npt}\,,
\qquad\ \npt
\in \left[-\frac{1}{2}\log T,+\infty\right).
\]
Given a point $p = (p_1,p_2) \in \R^2$ set also
\[
\widetilde\gamma^p(\npt) 
:= \frac{\gamma(t(\npt))-p}{\sqrt{2(T-t(\npt))}}\,, 
\qquad\ \npt
\in \left[-\frac{1}{2}\log T,+\infty\right).
\]
We let
$\widetilde I(\npt) := [0,L(\gammaresc(\npt))]$,
\begin{equation}\label{eq:nottilde}
\widetilde \tangvers(\npt) := \gammaresc_s(
\npt), 
\quad \widetilde \nu(\npt)
:= (-\widetilde \tangvers_2(\npt),
\widetilde \tangvers_1(\npt)) = \nu_{\gammaresc(\npt)}, \quad
\kappa_{\widetilde\gamma(\npt)} :=
\langle \frac{\gammaresc_{xx}(\npt)}{
\vert \gammaresc_x(\npt)
\vert^2}, \widetilde
\nu(\npt)\rangle, 
\end{equation}
$\widetilde \kappa(x,\npt) = \kappa_{\gammaresc(\npt)}(x)$,
and 
\begin{equation}\label{eqnottt}
\widetilde \lambda(\npt) :=
\langle \frac{\gammaresc_{xx}(\npt
)}{\vert \gammaresc_x(\npt
)\vert^2}, \widetilde
\tangvers(\npt)\rangle.
\end{equation}

Notice that $\widetilde\gamma$ satisfies the 
forced curvature flow equation
\begin{equation}\label{questaserviraforced}
\gammaresc_\npt
= \sqrt{2(T-t(\npt))}\, \gamma_t +\gammaresc
= \widetilde\kappa\widetilde\nu 
+ \widetilde\lambda\widetilde\tau + \gammaresc,
\end{equation}
coupled with the boundary conditions 
$\widetilde \gamma_2(0) = 
\widetilde \gamma_2(1) = \frac{-p_2}{\sqrt{2(T-t(\npt))}}$, and the 
usual Neumann boundary conditions 
\begin{equation}\label{lesolite}
\frac{\widetilde \gamma_x(0)}{\vert \widetilde \gamma_x(0)\vert} = 
\Big(
 \dfrac 1 2,\dfrac{\sqrt 3}{2}\Big), \qquad 
\frac{\widetilde \gamma_x(1)}{\vert \widetilde \gamma_x(1)\vert} = \Big(
\dfrac 1 2,- \dfrac{\sqrt 3}{2}\Big).
\end{equation}
As a consequence
by a direct computation (see \cite[formulae (2.7), (65), (66)]{MNT:04}) 
and using \eqref{questaserviraforced}
we get 
\begin{eqnarray}\label{questaservira}
\nonumber 
\widetilde\kappa_\npt
&=& \widetilde\kappa_{ss} 
+ \widetilde\lambda\widetilde\kappa_{s} + \left(\widetilde\kappa^2-1\right)\widetilde\kappa, 
\\
\widetilde\lambda_\npt 
&=& \widetilde\lambda_{ss}-\widetilde\lambda\widetilde\lambda_{s}-2\widetilde\kappa\widetilde\kappa_{s}
+ \left(\widetilde\kappa^2-1\right)\widetilde\lambda, 
\end{eqnarray}
Therefore, 
 letting 
$\widetilde w := \widetilde\kappa^2 + \widetilde\lambda^2$,
we find
\begin{equation}
\label{questaservirastr}
\begin{aligned}
\widetilde w_\npt
=& 
\widetilde w_{ss} - \widetilde\lambda\widetilde w_{s} 
+2 \left(\widetilde\kappa^2-1\right)\widetilde w
- 2 \left( \widetilde\kappa_s^2+\widetilde\lambda_s^2\right)
\\
\leq & 
\widetilde w_{ss} - \widetilde\lambda\widetilde w_{s} 
+2 \left(\widetilde\kappa^2-1\right)\widetilde w.
\end{aligned}
\end{equation}

In this section we  prove the 
following result, whose mainly follows 
the lines in \cite{MNT:04} (given for one triple
junction only), except for the arguments in step 8.

\begin{thm}\label{th:singuno}
Assume that $\inidat$ satisfies (A) and \eqref{asspos}. If $\gamma$
develops a type I singularity at $t=T$, then 
\begin{equation}\label{TE}
T = \frac{3\vert \regione(\inidat)\vert}{4\pi}, \qquad 
\lim_{t \to T^-} \vert \regione(\gamma(t))\vert =0,
\end{equation}
and
\begin{equation}\label{eqLbis}
\lim_{t\to T^-}L(\gamma(t))=0,
\end{equation}
so that $T$ is the extinction time of the evolution.
Moreover
\begin{itemize}
\item[-] there exists $t_c \in (0,T)$
such that $\gamma(t)$ is uniformly convex in $[0,1]$
for any $t \in [t_c,T)$;
\item[-] there exists  $p \in \R^2$ such that 
\begin{equation}\label{convgammah}
\lim_{\npt \to +\infty}\Vert \gammaresc^p(\npt)- \gamma^{{\rm h}} 
\Vert_{\C^2([0,1];\R^2)}=0.
\end{equation}
\end{itemize}  
\end{thm}

\begin{proof}
Let us assume that \eqref{eqsI} holds. 
From \cite[Th. 6.23]{MNT:04} it follows that, 
if we assume
\eqref{asspos}  and 
if in addition  $\inf_{t \in [0,T)}L(\gamma(t))>0$,
then $\gamma$  cannot develop type I singularities at $t=T$. 
Therefore
\begin{equation}\label{eq:liminf}
\liminf_{t\to T^-}
L(\gamma(t))=0.
\end{equation}
Using \eqref{eq:liminf} and the fact that $t \in [0,T) \to \vert
\regione(\gamma(t))\vert$ is decreasing (see
 Proposition \ref{prop:embedded} (ii)) 
it follows that $\lim_{t \to T^-} \vert \regione(\gamma(t))\vert=0$. 
In particular, from \eqref{eqvol} we have 
$T\le \frac{3 \vert \regione(\inidat)\vert}{4\pi}$, 
and the equality holds if and only if $\lim_{t \to T^-}
\vert \regione(\gamma(t))\vert =0$.
To prove \eqref{eqLbis}, we observe that, as in the proof of 
Proposition \ref{rem:lunghezza} and since the constant
$c$ in that statement is independent of $\gamma$, given $a,b \in (0,T)$ with $a<b$,
we have $L(\gamma(b)) \leq c L(\gamma(a))$, with 
$c>0$ independent of $a$ and $b$. This observation, coupled with  
\eqref{eq:liminf}, proves \eqref{eqLbis}.

{}From \eqref{eqLbis} 
and recalling the comparison argument used 
in the proof of Theorem \ref{th:exist}, we deduce that for any $x \in [0,1]$
there exists the limit 
$\lim_{t\to T^-}\gamma(x,t) \in \R^2$. Moreover,
by \eqref{eqLbis} such a limit is independent of $x$. 
We can therefore define
\begin{equation}\label{defp}
p:=
\lim_{t\to T^-}\gamma(x,t) \in \R^2.
\end{equation}
Set 
$$
\gammaresc
:= \gammaresc^p.
$$
Recalling the notation in \eqref{eq:nottilde}, thanks to \eqref{eqsI}
\begin{equation}\label{curvaequi}
|\widetilde\kappa(x,\npt
)|=\sqrt{2(T-t(\npt))}\,|\kappag(x,t(\npt
))|\le C, 
\qquad \npt \in \left[-\frac{1}{2} \log T, +\infty\right), \ 
x \in [0,1].
\end{equation}
We now divide the proof of the theorem into seven steps.

\medskip

\noindent
{\it Step 1}. 
We have 
\begin{equation}\label{fattouno}
\gammaresc(0,\npt),
\gammaresc(1,\npt) \in
 B_{\frac{2C}{\sqrt{3}}}(p),
\qquad \npt
 \in \left[-\frac{1}{2} \log T,+\infty\right),
\end{equation}
where 
$B_{\frac{2C}{\sqrt{3}}}(p)$ is the ball of radius
$\frac{2C}{\sqrt{3}}$ centered at $p$.

Indeed, since $-\kappag(0,\sigma) = \frac{\sqrt{3}}{2} \vert\gamma_t(0,\sigma)\vert$
for any $\sigma \in (0,T)$, 
using \eqref{curvaequi} we have 
\begin{eqnarray*}
 |\gammaresc(0,\npt
)| &=& 
\frac{1}{\sqrt{2(T-t(\npt
))}}\vert
\int_{t(\npt)}^T \gamma_t(0,\sigma)\,d\sigma \vert
\\
&\le& \frac{2}{\sqrt 3\,\sqrt{2(T-t(\npt))}} \int_{t(\npt)}^T 
\left| \kappag(0,\sigma)\right|\,d\sigma
\\
&\leq& \frac{2C}{\sqrt{3}\sqrt{2(T- t(\npt
))}}
\int_{t(\npt
)}^{T} 
\frac{1}{\sqrt{2(T-\sigma)}}~d\sigma 
= 
\frac{2C}{\sqrt{3}}.
\end{eqnarray*}
Since the same estimate holds for $|\gammaresc(1,\npt
)|$,
step 1 is proved.

\medskip

\noindent
{\it Step 2}. We have
\begin{equation}\label{fattodue}
\vert \regione(\gammaresc(\npt
))\vert  = \frac{4\pi}{3},\qquad
\npt
 \in 
\left[-\frac{1}{2} \log T,+\infty\right).
\end{equation} 

Indeed, from \eqref{eqvol} 
and \eqref{TE} 
it follows that 
$\vert \regione(\gamma(t))\vert = \frac{4\pi}{3}(T-t)$, and therefore
\eqref{fattodue} follows from the definition of $\gammaresc$.

\medskip

Without loss of generality, from now on we  assume $p=(0,0)$. 
We recall
the so-called {\it rescaled monotonicity formula} 
(see \cite{Hu:90}, \cite[Prop. 6.7]{MNT:04}):
\begin{equation}\label{monotone}
\frac{d}{d \npt
}\,\int_{\widetilde I(\npt
)} 
e^{-\frac{|\gammaresc(\npt
)|^2}{2}}\,ds
= 
 - \int_{\widetilde I(\npt
)} 
e^{-\frac{|\gammaresc(\npt
)|^2}{2}}\,\left| 
\kappa_{\gammaresc(\npt
)} + 
\langle \gammaresc(\npt
),\nu_{\gammaresc(\npt
)}\rangle\right|^2\,ds =: - f(\npt) 
\,\le\,0\,.
\end{equation}
Integrating \eqref{monotone} on $[-\frac{1}{2}\log T,+\infty)$ we get
\[
\begin{aligned}
\int_{-\frac{1}{2}\log T}^{+\infty}
f(\npt) ~d\npt
 = & 
\int_{\widetilde I(-\frac{1}{2} \log T)}
e^{-\frac{\vert \gammaresc(-\frac{1}{2} \log T)\vert^2}{2}} ~ds 
= \frac{1}{\sqrt{2T}} \int_{I(0)}
e^{-\frac{\vert \gamma(0)\vert^2}{4 T}} ~ds <+\infty.
\end{aligned}
\]
As a consequence, the nonnegative  function $f$
belongs to $L^1([-\frac{1}{2} \log T, +\infty))$.
Since $\sum_{j=1}^{+\infty} \frac{1}{j} = +\infty$,
we then have 
that for any sequence $\{\npt
_n\} \subset (-\frac{1}{2} \log T, +\infty)$ 
converging to $+\infty$, 
there exist a subsequence $\{\npt
_{n_j}\}$ and times $r_j\in 
[\npt
_{n_j},\npt
_{n_j}+1/j]$ such that 
\begin{equation}\label{rlim}
\lim_{j\to+\infty} f(r_j)
=0.
\end{equation}

\medskip

Assume now that 
\begin{equation}\label{fattotre}
\sup_{\npt
 \in 
[-\frac{1}{2} \log T,+\infty)} L(\gammaresc(\npt
)) < 
+\infty.
\end{equation} 

\medskip

\noindent
{\it Step 3}. Weak convergence to $\gamma^\infty$ in $W^{2,\infty}$ along 
a subsequence $\{r_{j^{}_k}\}$.

{}From \eqref{curvaequi} and assumption \eqref{fattotre} we have that 
\[
\sup_j \left[L(\gammaresc(r_j)) 
+ \Vert \kappa_{\gammaresc(r_j)}\Vert_{L^\infty([0,1])}\right] < +\infty.
\]
It follows that there exist a subsequence $\{r_{j^{}_k}\}$
and a map 
\begin{equation}\label{maplim}
\gamma^\infty \in W^{2,\infty}([0,1];\R^2),
\end{equation}
such that $\gammaresc({r_j}_k)$ converges to $\gamma^\infty$
weakly in $W^{2,\infty}([0,1];\R^2)$ as $k \to +\infty$. In particular
\begin{equation}\label{manu}
\lim_{k \to +\infty}
\Vert 
\gammaresc({r_j}_k) - \gamma^\infty
\Vert_{
\C^1([0,1]; \R^2)
}
=0,
\end{equation}
and 
\begin{equation}\label{redd}
\lim_{k \to +\infty} \gammaresc_{xx}(r_{j^{}_k}) = \gamma^\infty_{xx} \qquad
{\rm weakly~ in~} L^2([0,1];\R^2).
\end{equation}
Hence from steps 1,2,3 and \eqref{fattotre} it follows that 
\begin{itemize}
\item[]
\item[(i)] 
$\gamma^\infty(0), \gamma^\infty(1) 
\in B_{\frac{2C}{\sqrt{3}}}(0)$, and 
$\gamma^\infty(0), \gamma^\infty(1)$ belong to the first coordinate
axis;
\item[]
\item[(ii)] 
$\frac{\gamma^\infty_x(0)}{\vert \gamma^\infty_x(0)\vert} = 
\Big( \dfrac 1 2,\dfrac{\sqrt 3}{2}\Big)$, 
$\frac{\gamma_x^\infty(1)}{\vert \gamma_x^\infty(1)\vert} = 
\Big(\dfrac 1 2,-\dfrac{\sqrt 3}{2}\Big)$;
\item[]
\item[(iii)] $\vert \regione(\gamma^\infty)\vert = \frac{2\pi}{3}$;
\item[]
\item[(iv)] $L(\gamma^\infty) < +\infty$.
\end{itemize}
\medskip 
Moreover, as a consequence of (iii), and respectively of 
(ii), (iv) and \eqref{maplim}, we have
\begin{itemize}
\item[(v)] $L(\gamma^\infty) >0$;
\item[(vi)] $\kappa_{\gamma^\infty}$ is not 
identically zero.
\end{itemize}

\medskip

\noindent
{\it Step 4}. We have 
\begin{itemize}
\item[]
\item[(vii)] $\gamma^\infty_2(x) >0$ for any $x \in (0,1)$;
\item[(viii)]
$\gamma^\infty$ is injective.
\end{itemize}
Indeed, from \eqref{eq:invresc} and \eqref{eq:gmon} we have 
\begin{equation}\label{ohai}
g_{\gammaresc} (\npt
) = g_\gamma(t(\npt))\geq 
\min(g_\gamma(0), 4 \sqrt{3}), \qquad \npt \in \left[-\frac{1}{2} \log T, +\infty
\right).
\end{equation}
Moreover, since $g_{\gammaresc}(\npt)$ is defined as an 
infimum, it is upper semicontinuous, in the sense that
\begin{equation}\label{quant}
\lim_{k \to +\infty} 
\Vert \gammaresc(r_{j^{}_k}) - 
\gamma^\infty\Vert_{\C^1([0,1];\R^2)} =0
\Rightarrow 
g_{\gamma^\infty}
\geq \limsup_{k \to 
+\infty}
g_{\gammaresc} (r_{j^{}_k}),
\end{equation}
where $g_{\gamma^\infty}$ is (the constant) defined as in \eqref{def:g}, 
where we substitute $\gamma(\cdot,t)$ with $\gamma^\infty(\cdot)$ 
on the right hand side of \eqref{defQ1}.
{}From \eqref{ohai} and \eqref{quant} it follows that 
$g_{\gamma^\infty} \geq \min(g_\gamma(0), 4 \sqrt{3})$,
and this implies (vii) and (viii).

\medskip
As a consequence of (ii) and (viii) we have:

\begin{itemize}
\item[(ix)] $\gamma^\infty(0) \neq  \gamma^\infty(1)$.
\end{itemize}

\medskip

\noindent
{\it Step 5}. We have 
\begin{equation}\label{omofob} 
\kappa_{\gamma^\infty} + \langle \gamma^\infty,\nu_{\gamma^\infty}\rangle\,=0 
\qquad  {\rm a.e. ~in}~ [0,1].
\end{equation} 
Indeed, from Fatou's Lemma and \eqref{rlim} we have 
\begin{equation}\label{semi}
\int_0^1 \liminf_{j \to +\infty}
\left[
e^{-\frac{|\gammaresc(r_j)|^2}{2}}\,\left| 
\kappa_{\gammaresc(r_j)} + 
\langle \gammaresc(r_j),\nu_{\gammaresc(r_j)}\rangle\right|^2
\vert \gammaresc_x(r_j)\vert\right]\,dx
\leq \lim_{j \to +\infty} f(r_j) =0.
\end{equation}
On the other hand, 
by \eqref{manu} and \eqref{redd},
the left hand side of \eqref{semi} equals
\begin{equation}\label{cont}
\int_0^{L(\gamma^\infty)} 
e^{-\frac{|\gamma\infty|^2}{2}}\,\left| 
\kappa_{\gamma^\infty} + 
\langle \gamma^\infty,\nu_{\gamma^\infty}\rangle\right|^2
\,ds,
\end{equation}
and \eqref{omofob} follows.

\medskip 

By elliptic regularity \cite{GiTr:83} it follows that
 $\kappa_{\gamma^\infty}
\in \C^0([0,1])$, hence $\gamma^\infty \in \C^2([0,1];\R^2)$, and 
\eqref{omofob} is valid
everyhere in classical sense in $[0,1]$. 
Recalling 
Section \ref{rem:homotetico}, we deduce by uniqueness that 
\begin{equation}\label{lim:identified}
\gamma^\infty = \gamma^{{\rm 
h}}.
\end{equation}
Note that from \eqref{lim:identified} it follows that $\gamma^\infty$ 
is independent of the subsequence $\{j^{}_k\}$, hence \eqref{manu}
is valid for the whole sequence $\{r_j\}$, i.e. 
\begin{equation}\label{manuj}
\lim_{j \to +\infty}
\Vert 
\gammaresc({r_j}) - \gamma^\infty
\Vert_{
\C^1([0,1]; \R^2)
}
=0.
\end{equation}

\noindent
{\it Step 6}.
We have
\begin{equation}\label{stepp}
\lim_{j \to +\infty}
\Vert
\gammaresc(\npt_{n_j}) - \gamma^{{\rm h}}\Vert_{\C^1([0,1];\R^2)}=0.
\end{equation}
{}From step 3
applied to the sequence $\{\gammaresc(\npt_{n_j})\}$ in place
of $\{\gammaresc(r_j)\}$, it follows that there exist
a map 
$$
\gammaresc^\infty \in W^{2,\infty}([0,1]; \R^2)
$$
and a subsequence $\{n_{j^{}_h}\}$ 
such that 
\begin{equation}\label{lisar}
\lim_{h \to +\infty}
\Vert 
\gammaresc(\nuovoparametrotemporale_{n_{j^{}_h}}) - \gammaresc^\infty
\Vert_{
\C^1([0,1]; \R^2)
}
=0,
\end{equation}
and
such that  $\gammaresc^\infty$ satisfies 
properties (i)-(ix) listed in steps 3,4.

In order to show \eqref{stepp}, it is enough to 
prove that 
\begin{equation}\label{casi}
\gammaresc^\infty = \gamma^{{\rm h}}.
\end{equation}
Using \eqref{lisar}, \eqref{manuj} and the inequality
\begin{equation*}
\begin{aligned}
\Vert \gammaresc^\infty - \gamma^{{\rm h}}
\Vert_{\C^0([0,1];\R^2)}
 \leq
~ &  
\Vert \gammaresc^\infty - 
\gammaresc(\npt_{n_{j^{}_h}})
\Vert_{\C^0([0,1];\R^2)}  
+ 
\Vert 
\gammaresc(\npt_{n_{j^{}_h}})
- 
 \gammaresc(r_{j^{}_h})\Vert_{\C^0([0,1];\R^2)} 
\\
& +
\Vert
\gammaresc(r_{j^{}_h}) - \gamma^{{\rm h}}\Vert_{\C^0([0,1];\R^2)},
\end{aligned}
\end{equation*}
to prove \eqref{casi} it is sufficient to  show that
\begin{equation}
\label{monteporzio}
\lim_{j\to +\infty} 
\Vert \gammaresc(r_j) - \gammaresc(\npt_{n_j})\Vert_{\C^0([0,1];\R^2)} =0.
\end{equation}
In order to prove \eqref{monteporzio},
we recall that $\widetilde\kappa(x,\npt)$ is uniformly bounded for all $(x,
\npt)$ by \eqref{curvaequi} 
and, as a consequence, $\widetilde\lambda(x,\npt)$ is also uniformly bounded 
by \eqref{questaservira} and \eqref{questaservirastr} as in 
\cite[p. 264]{MNT:04}. 
Hence, using also \eqref{questaserviraforced} and \eqref{monotone}, 
\begin{equation*}
\begin{aligned}
\Vert \gammaresc(r_j) - 
\gammaresc(\npt_{n_j})\Vert_{\C^0([0,1];\R^2)} \le & 
\int_{\npt_{n_j}}^{r_j} \int_0^1 \vert\gammaresc_\npt\vert\, dx
\le \int_{\npt_{n_j}}^{r_j} \int_0^1 \left(\vert\widetilde \kappa\vert 
+  \vert\widetilde \lambda\vert + \vert \gammaresc\vert\right)\,dx
\\
\le & C\,\vert r_j-\npt_{n_j}\vert
\le \frac C j\,,
\end{aligned}
\end{equation*}
which gives \eqref{monteporzio} and proves step 6.

\medskip

{}From \eqref{stepp} and \cite[Prop. 6.16]{MNT:04} 
we have the improved
convergence
\begin{equation}\label{improved}
\lim_{j \to +\infty}
\Vert 
\gammaresc(\npt_{n_j}) - \gamma^{{\rm h}}
\Vert_{
\C^2([0,1];\R^2)
}
=0.
\end{equation}
\noindent Since the sequence $\{\npt_{n}\}$ is arbitrary we deduce
\begin{equation}\label{convgammahuno}
\lim_{\npt \to +\infty}\Vert \gammaresc(\npt)- \gamma^{{\rm h}}
\Vert_{\C^2([0,1];\R^2)}=0.
\end{equation}

\smallskip
Eventually, we observe 
that, since $\gamma^\infty_2$ is uniformly concave in $[0,1]$
(see Section \ref{rem:homotetico}), from \eqref{improved} 
we deduce that 
$\gamma(t_c)$ becomes uniformly convex for some $t_c \in (0,T)$. 
{}From the results proved in  \cite[Lemma 3.3]{SS:08} it follows 
that $\gamma(t)$ remains uniformly
convex in $[t_c,T)$
(this 
last assertion also follows
from \eqref{eqcurv} and 
\eqref{eqkappasdiffdue} using the maximum principle).

\medskip

\noindent
{\it Step 7}.
Assume now that \eqref{fattotre} does not hold, that is, there exists a 
sequence $\{\npt_n\}$ converging to $+\infty$ such that 
\begin{equation}\label{cud}
\lim_{n\to +\infty} L(\gammaresc(\npt_n)) = +\infty.
\end{equation}
Resoning as in 
step 1, there exist a subsequence $\{\npt_{n_j}\}$ 
and times $r_j\in [\npt_{n_j},\npt_{n_j}+1/j]$ such that \eqref{rlim} holds.
Moreover
from \eqref{curvaequi} 
and $r_j-\npt_{n_j}\le 1/j$, and from \eqref{jelka} and \eqref{cud} we obtain 
\begin{equation}\label{caf}
\lim_{j\to +\infty} L(\gammaresc(r_j)) = +\infty.
\end{equation}
If we
parametrize $\gammaresc(r_j)$ by arclength on $[0,L(\gammaresc(r_j))]$,
and we pass to the limit as in step 3 
as $j\to +\infty$, 
we get that there exists a subsequence $\{r_{j^{}_k}\}$ such that 
$\{\gammaresc(r_{j^{}_k})\}$
converges weakly in 
$W^{2,2}_{{\rm loc}}([0,+\infty);\R^2)$
(so that \eqref{manu} and \eqref{redd} hold with 
$\mathcal C^1_{\rm loc}([0,+\infty);\R^2)$ in place of 
$\mathcal C^1([0,1];\R^2)$ and 
$L^2_{{\rm loc}}([0,+\infty);\R^2)$ in place of
$L^2([0,1];\R^2)$
respectively)
to a curve $\gamma^\infty$ of infinite length which,
arguing as in steps 4,5, has  the 
following properties:
\begin{itemize}
\item[(a)] $\gamma^\infty_2(0)=0$, $\gamma_2^\infty(s) > 0$ 
for any $s \in (0,+\infty)$;
\item[(b)] $\gamma^\infty_s(0) = 
\Big( \dfrac 1 2,\dfrac{\sqrt 3}{2}\Big)$;
\item[(c)] $\gamma^\infty$ is injective (by using \eqref{ohai});
\item[(d)] $\gamma^\infty$ 
solves 
\eqref{omofob} almost everywhere in $[0,+\infty)$. 
\end{itemize}
By elliptic regularity $\gamma^\infty \in 
\C^\infty([0,+\infty); \R^2)$ 
and solves 
\eqref{omofob} in classical sense. Then 
by the  results in \cite{CG:07} and \cite{SS:08} it follows that
 $\gamma^\infty([0,+\infty)$
is contained in a curve of Abresch-Langer \cite{AbLa:86}. 
In view of the Neumann condition (b) and the properties of the curves of 
Abresch-Langer,
it then follows that
\[
\gamma^\infty(s) = \left( \frac{s}{2}\,,\frac{\sqrt 3 s}{2}\right),
\qquad s\in [0,+\infty).
\]
Similarly, if we parametrize $\gammaresc(r_j)$ 
by arclength on $[-L(\gammaresc(r_j)),0]$,
we find a  subsequence 
$\{\widetilde \gamma(r_{j_{k_\ell}})\}$ of 
$\{\widetilde \gamma(r_{j_k})\}$
converging to a curve
$\gamma^\infty \in \C^\infty((-\infty,0];\R^2)$ 
of infinite length satisfying 
(a), (c), (d), $\gamma_s^\infty(0) = 
(\frac{1}{2}, - \frac{\sqrt{3}}{2})$, and 
contained in a curve of Abresch-Langer. Hence necessarily
\[
\gamma^\infty(s) 
= \left(\frac{s}{2}, -\frac{\sqrt 3 s}{2}\right),
\qquad s\in (-\infty,0].
\]
We now reach a contradiction
since, being the convergence of 
$\{\gammaresc(r_{j^{}_{k_\ell}})\}$ in $\mathcal C^1_{\rm loc}$,
it follows that 
$\gammaresc(r_{j^{}_{k_\ell}})$ 
is not injective for $\ell$ sufficiently large.
Indeed, provided $\ell \in \mathbb N$  is such that
\[
\| \widetilde\gamma(s,r^{}_{j_{k_\ell}})
- \gamma^\infty(s)\|_{\C^1([0,1];\R^2)}
+
\| \widetilde\gamma\left(L(\widetilde\gamma(r^{}_{j_{k_\ell}}))
-s,r^{}_{j_{k_\ell}}\right)
- \widetilde\gamma^\infty(-s)\|_{\C^1([0,1];\R^2)}
\le \frac{1}{2},
\]
recalling the boundary conditions in \eqref{lesolite},
we have that there exist $s_1,s_2\in [0,1]$
such that
\[
\widetilde\gamma(s_1,r^{}_{j_{k_\ell}}) 
= \widetilde\gamma
\left(L(\widetilde\gamma(r^{}_{j_{k_\ell}}))-s_2,r^{}_{j_{k_\ell}}\right).
\]
Hence \eqref{fattotre} necessarily holds, and the proof of the theorem
is complete.
\end{proof}

\section{Embedded nonconvex initial data: type II singularities}
\label{secpostwo}

This section is devoted to the proof of the following result.

\begin{thm}\label{teosingII}
Assume that $\inidat$ satisfies (A) and \eqref{asspos}.
Then 
$\gamma$ cannot develop type II singularities at $t=T$. 
\end{thm}
\begin{proof}
Let us assume by contradiction that 
$\gamma$ develops a type II singularity at $t=T$. 
We employ a rescaling procedure 
originally due to R. Hamilton (see \cite{Al:91}).
Let us choose as in  \cite[Section 7.1]{MNT:04}
a sequence $\{(x_n,t_n)\} \subset [0,1] \times 
[0,T)$ 
satisfying the following properties:
\begin{itemize}
\item[-] $t_n\in [0,T-1/n)$ and 
$t_n < t_{n+1}$ 
for any $n \in \mathbb N$;
\item[-] letting
$$
\paraII_n := |\kappa_{\gamma}(x_n,t_n)|, \qquad n \in \mathbb N,
$$
we have $0 < \paraII_n < \paraII_{n+1}$ and $\lim_{n \to +\infty}
\paraII_n = +\infty$;
\item[-] 
\begin{equation}\label{att}
\lim_{n \to +\infty}
\paraII_n \sqrt{T-1/n-t_n} = +\infty,
\end{equation}
and for any $n \in \mathbb N$ 
\begin{equation}\label{eqeq}
\paraII_n \sqrt{T-1/n-t_n} = 
\max_{t\in [0,T-1/n]}\left(
\| \kappa_{\gamma(t)}\|_{L^\infty([0,1])} ~
\sqrt{T-1/n-t}\right).
\end{equation}
\end{itemize}
Note that the maximum in \eqref{eqeq} is attained in $[0,T- 1/n)$
by \eqref{att}. Note also that
\begin{equation}\label{notethatthemax}
\lim_{n \to +\infty} -\paraII_n^2 t_n = -\infty, 
\qquad
\lim_{n \to +\infty} \paraII_n^2 (T-t_n) = +\infty.
\end{equation}
Let us define the 
parameter 
$\nuovoparametrotemporale$ as 
$$
t(\nuovoparametrotemporale) := t_n + \nuovoparametrotemporale/\paraII_n^2,
\qquad 
\nuovoparametrotemporale
\in [-\paraII_n^2 t_n, \paraII_n^2(T-t_n)),
$$
and
the  curves $\gamma_n$ as 
\[
\gamma_n(x,\nuovoparametrotemporale) := \paraII_n\Big( 
\gamma\left(x,t(\nuovoparametrotemporale)\right) - 
\gamma(x_n,t_n)\Big), \qquad x \in [0,1], \ 
\nuovoparametrotemporale
\in [-\paraII_n^2 t_n, \paraII_n^2(T-t_n))\,.
\]
We have
\begin{equation}\label{kappanorm}
\gamma_n(x_n,0) = (0,0), \qquad |\kappa_{\gamma_n}
(x_n,0)| = 1, \qquad n \in \mathbb N.
\end{equation}
{}From \eqref{eqeq} it follows as in \cite[Sec. 7]{MNT:04} that 
for every $\varepsilon,\omega>0$ there exists $\overline n\in
\mathbb N$ such that 
\begin{equation}\label{vedi}
\| \kappa_{\gamma_n(\nuovoparametrotemporale)}
\|_{L^\infty([0,1])}\le 1
+\varepsilon, \qquad 
n\ge\overline n, \ \ 
\nuovoparametrotemporale\in 
[ -\paraII_n^2 t_n,\omega].
\end{equation}
We now divide the proof of the theorem into nine steps.

\medskip

\noindent
{\it Step 1}. We have 
\begin{equation}
\label{eq:lengthinf}
\lim_{n \to +\infty}L(\gamma_n(\nuovoparametrotemporale))=+\infty,
\qquad \nuovoparametrotemporale \in \R.
\end{equation}
Indeed, this is obvious if $T$ is not the extinction time,
since in that case $\inf_{t \in [0,T)}L(\gamma(t))>0$.
If $T$ is the extinction time, namely  $0=\lim_{t\to T^-}
L(\gamma(t)) = \lim_{t \to T^-} \vert \regione(\gamma(t))\vert$, 
by the isoperimetric inequality and taking into account that 
$\gamma$ satisfies  
\eqref{ipo:datum}, it follows that 
there exists 
an absolute constant $c>0$ 
such that $L(\gamma(t)) \geq c 
\sqrt{\vert \regione(\gamma(t))\vert}$ for all $t\in [0,T)$.
Hence, to prove \eqref{eq:lengthinf}   
it is enough
to show that 
\begin{equation}\label{volinf}
\lim_{n \to +\infty}
\vert \regione(\gamma_n(\nuovoparametrotemporale))\vert = +\infty, 
\qquad 
\nuovoparametrotemporale\in \R.
\end{equation}
Recalling \eqref{eqvol}, we have 
\[
\vert \regione(\gamma_n(\nuovoparametrotemporale))\vert = \paraII_n^2 \vert 
\regione(\gamma(t(\nuovoparametrotemporale)))\vert 
= \frac{4}{3}\pi \paraII_n^2 \left(T-t(\nuovoparametrotemporale)\right),
\qquad 
\nuovoparametrotemporale \in [-\paraII_n^2 t_n, \paraII_n^2 (T-t_n)).
\]
In particular $\vert \regione(\gamma_n(0))\vert 
= \frac{4}{3}\pi\paraII_n^2 (T-t_n)$, hence 
$\lim_{n \to +\infty}
\vert \regione(\gamma_n(0))\vert = +\infty$ by \eqref{notethatthemax}.
Then step 1 follows,
since 
$
\vert 
\regione(\gamma_n(\nuovoparametrotemporale))\vert = 
\vert 
\regione(\gamma_n(0))\vert  
- \frac{4}{3}\pi \nuovoparametrotemporale$ for any 
$\nuovoparametrotemporale \in 
[-\paraII_n^2 t_n, \paraII_n^2 (T-t_n))$.
\medskip

Before passing to the next step we need some preparation.
Given $\nuovoparametrotemporale \in [-\paraII_n^2 t_n, \paraII_n^2 (T-t_n))$, 
we now reparametrize the curves $\gamma_n(\nuovoparametrotemporale)$ 
by arclength 
and, performing a suitable translation 
in the parameter space, we obtain curves 
$$
\gammaresctipoII_n(\nuovoparametrotemporale): 
 [a_n(\nuovoparametrotemporale),b_n(\nuovoparametrotemporale)] 
\to \R^2,
$$
with 
$a_n(\nuovoparametrotemporale)\le 0\le b_n(\nuovoparametrotemporale)$, 
and $b_n(\nuovoparametrotemporale)- a_n(\nuovoparametrotemporale) = 
L(\gamma_n(\nuovoparametrotemporale))$. 

Thanks to \eqref{eq:lengthinf}, we have 
\begin{equation}\label{astutoc}
\lim_{n \to +\infty}
\Big(
b_n(\nuovoparametrotemporale)-a_n(\nuovoparametrotemporale)
\Big)
= +\infty, \qquad 
\nuovoparametrotemporale\in \R.
\end{equation}
Without loss of generality we  assume 
\begin{equation}\label{cincin}
\gammaresctipoII_n(0,0)=\gamma_n(x_n,0)=(0,0).
\end{equation}
We can also assume that there exists a subsequence $\{n_j\}$ such that 
\begin{equation}\label{fixarc}
\lim_{j \to +\infty}
a_{n_j}(0) =: a_\infty \in [-\infty, 0],
\qquad
\lim_{j \to +\infty} b_{n_j}(0) =: b_\infty 
\in [0, +\infty].
\end{equation}
Note that by \eqref{astutoc} we have that if $a_\infty 
\in (-\infty,0]$ (resp. $b_\infty \in [0,+\infty)$) then 
$b_\infty = +\infty$ (resp. $a_\infty = - \infty$).

We now choose the starting point of the reparametrization 
(still keeping the notation $\widehat \gamma_n$)
as follows.
If $b_\infty = +\infty$
we set 
$a_{n_j}(\nuovoparametrotemporale) := a_{n_j}(0)$ for any 
$\nuovoparametrotemporale \in \R$; 
if $b_\infty  \in [0,+\infty)$ we set 
$b_{n_j}(\nuovoparametrotemporale) := b_{n_j}(0)$ for any 
$\nuovoparametrotemporale \in \R$.
Hence in both cases
\begin{equation}\label{fixarcnew}
\lim_{j \to +\infty}
a_{n_j}(\nuovoparametrotemporale) =: a_\infty,
\qquad
\lim_{j \to +\infty} b_{n_j}(\nuovoparametrotemporale) =: b_\infty, 
\qquad \nuovoparametrotemporale \in \R.
\end{equation}
 If $a_\infty \in (-\infty,0]$ (resp. $b_\infty \in [0,+\infty)$) we set
$I_\infty :=  [a_\infty, +\infty)$ (resp.
$I_\infty :=  (-\infty, b_\infty]$);
if $\vert a_\infty\vert = b_\infty = +\infty$ we set $I_\infty := \R$. 
Observe that $0 \in I_\infty$.

\medskip

Exploiting also \eqref{cincin}, 
the proof of the next step 
is the same as  in \cite[Prop. 7.1]{MNT:04}, using also 
\eqref{astutoc}, \eqref{vedi} and \eqref{kappanorm}.

\noindent
{\it Step 2}. 
The sequence 
$\{\gammaresctipoII_{n_j}\}$  admits a subsequence
$\{\gammaresctipoII_{n_{j^{}_h}}\}$ 
converging in $\C^2_{\rm loc}(I_\infty\times \R;\R^2)$ to an 
embedded 
curvature evolution $\gamma_\infty \in \C^\infty(I_\infty \times \R; \R^2)$
with 
\begin{eqnarray}
L(\gamma_\infty(\nuovoparametrotemporale)) &=& +\infty,
\qquad \nuovoparametrotemporale \in \R,
\nonumber
\\
\gamma_\infty(0,0) &=& (0,0),
\nonumber 
\\
\label{siconciau}
\Vert \kappa_{\gamma_\infty}\Vert_{L^\infty(I_\infty \times \R)} &=&  1
= \vert \kappa_{\gamma_\infty}(0,0)\vert.
\end{eqnarray}
Moreover
\begin{itemize}
\item[-] 
if $I_\infty = [a_\infty,+\infty)$ 
then
${\gamma_\infty}_s(a_\infty,\nuovoparametrotemporale) = (1/2,\sqrt 3/2)$
for all $\nuovoparametrotemporale \in \R$, and
$$
{\gamma_\infty}_2(s,\nuovoparametrotemporale)
\ge {\gamma_\infty}_2(a_\infty,\nuovoparametrotemporale),
\qquad s \in I_\infty, ~
\nuovoparametrotemporale\in\R;
$$
\item[-]
if
$I_\infty = (-\infty,b_\infty]$
then
${\gamma_\infty}_s(b_\infty,\nuovoparametrotemporale) = (1/2,-\sqrt 3/2)$ 
for all $\nuovoparametrotemporale\in\R$, and
$$
{\gamma_\infty}_2(s,\nuovoparametrotemporale)\ge 
{\gamma_\infty}_2(b_\infty,\nuovoparametrotemporale), 
\qquad s \in I_\infty, ~
\nuovoparametrotemporale\in\R.
$$
\end{itemize}

\medskip

Note that the $\C^2_{{\rm loc}}(I_\infty\times\R; \R^2)$-convergence 
can be improved to $\C^\infty_{{\rm loc}}(I_\infty\times \R; \R^2)$
\cite{GH:86}, since the curves $\widehat \gamma_n$ evolve by curvature
and have a uniform $L^\infty$-bound on their curvature.

\medskip

\noindent
{\it Step 3}. 
For all $\nuovoparametrotemporale \in \R$ we have 
$\kappa_{\gamma_\infty}(s,\nuovoparametrotemporale)\neq 0$ for all
$s \in I_\infty$.

We follow \cite[Th. 7.7]{Al:91}.
Write
for simplicity 
$$
J_h(\nuovoparametrotemporale) := [
a_{n_{j^{}_h}}(\nuovoparametrotemporale), 
b_{n_{j^{}_h}}(\nuovoparametrotemporale) 
], \qquad 
\widehat \kappa_h(s,\nuovoparametrotemporale) = \kappa_{
\widehat 
\gamma_{n_{j^{}_h}}}
(s, \nuovoparametrotemporale), \qquad
z_h(\nuovoparametrotemporale):= 
\left\{s \in J_h(\nuovoparametrotemporale):\,\widehat \kappa_h
(s,\nuovoparametrotemporale
)=0\right\}.
$$
For all $M>0$, 
recalling \eqref{eq:absvalkappa}, we have 
\begin{equation}\label{stanc}
\begin{aligned}
-2\int_{-M}^M 
\sum_{
s \in z_h(\nuovoparametrotemporale)
}
|\partial_s \widehat \kappa_h |
\,d\nuovoparametrotemporale
&=
\int_{-M}^M \frac{d}{d \nuovoparametrotemporale}
\int_{
J_h(\nuovoparametrotemporale)
}
|\widehat\kappa
_h|\,ds\,d\nuovoparametrotemporale 
\\
&=
\int_{
J_h(M)
}|\widehat \kappa_h(s,M)|\,ds - \int_{
J_h(-M)
}|\widehat \kappa_h(s,-M)|\,ds.
\end{aligned}
\end{equation}
Using the invariance of  $\int_{I(t)} 
|\kappa_\gamma(\cdot,t)|~ds$ 
under rescalings and writing
$$
\gamma_h := \gamma_{n_{j^{}_h}}, \qquad 
t_h := t_{n_{j^{}_h}}, \qquad \paraII_h := \paraII_{n_{j^{}_h}}, 
$$
{}from \eqref{stanc} we then obtain
\begin{equation}\label{otti}
\begin{aligned}
& -2
\int_{-M}^M \sum_{s \in z_h(\nuovoparametrotemporale
)}|\partial_s \widehat
\kappa_h |\,d\nuovoparametrotemporale
\\
= &
\int_{
I\big(
t_h 
+
\frac{M}{\paraII_h^2}
\big)}
|\kappa_{\gamma_h}
(s, t_h 
+
M/\paraII_h^2
)
|\,ds - \int_{I\big(t_h 
-\frac{M}{\paraII_h^2}\big)}
|
\kappa_{\gamma_h}
(s, t_h 
-
M/\paraII_h^2
)
|
\,ds. 
\end{aligned}
\end{equation}
In view of 
Proposition \ref{prop:tripodi} the function $t\to \int_{I(t)}
|\kappa_{\gamma(t)}|\,ds$ is
nonincreasing, hence it admits
a  finite limit as $t \to T^-$. In particular, 
$$
\lim_{h \to +\infty}
\int_{
I\left(
t_h 
+\frac{M}{\paraII_h^2}\right)}
|\kappa_{\gamma_h}(s, t_h 
+
M/\paraII_h^2
)
|\,ds =
\lim_{h \to +\infty}
 \int_{I\left(t_h 
-
\frac{M}{\paraII_h^2}\right)}
|
\kappa_{\gamma_h}
(s, t_h 
-
M/\paraII_h^2
)
|
\,ds. 
$$
It then follows from \eqref{otti} that 
\begin{equation}\label{bisco}
\lim_{h \to +\infty}
\int_{-M}^M \sum_{s \in z_h(\nuovoparametrotemporale
)}|\partial_s \widehat
\kappa_h |\,d\nuovoparametrotemporale
=0.
\end{equation}
{}From \eqref{bisco} and Fatou's Lemma we deduce that 
\begin{equation}\label{dil}
0 = \liminf_{h\to+\infty}\sum_{s \in z_h
(\nuovoparametrotemporale)}|\partial_s \widehat\kappa_h 
(s,\nuovoparametrotemporale)
|
\qquad 
{\rm for ~ a.e.~} 
\nuovoparametrotemporale
\in [-M,M].
\end{equation}
Since \eqref{dil} holds for any $M>0$, 
and all quantities involved are continuous
with respect to $\nuovoparametrotemporale$, we obtain 
\begin{equation}\label{bianca}
0 = \liminf_{h\to+\infty}\sum_{s \in z_h
(\nuovoparametrotemporale)}|\partial_s \widehat\kappa_h
(s,\nuovoparametrotemporale)
|,
\qquad \nuovoparametrotemporale \in \R.
\end{equation}
On the other hand, the $\C^2_{{\rm loc}}(I_\infty\times \R; \R^2)$-convergence of 
$\widehat\gamma_h$ to $\gamma_\infty$ 
given in step 2 implies that 
\begin{equation}\label{latte}
\liminf_{h\to+\infty}\sum_{s \in z_h
(\nuovoparametrotemporale)}|\partial_s \widehat\kappa_h
(s,\nuovoparametrotemporale)
|
 \ge 
\sum_{s \in I_\infty:\,\kappa_{\gamma_\infty}
(s,\nuovoparametrotemporale)=0}|\partial_s \kappa_{\gamma_\infty}
(s,\nuovoparametrotemporale)
|, 
\qquad \nuovoparametrotemporale \in \R.
\end{equation}
Since the right hand side of \eqref{latte} is nonnegative,
from \eqref{bianca} we deduce 
$$
0=\sum_{s \in I_\infty:\,
\kappa_{\gamma_\infty}(s,\nuovoparametrotemporale)
=0}|\partial_s \kappa_{\gamma_\infty}
(s,\nuovoparametrotemporale)
|, 
\qquad \nuovoparametrotemporale \in \R.
$$
It follows that for any $\nuovoparametrotemporale \in \R$ we have
$$
\Big\{
s \in I_\infty : 
\kappa_{\gamma_\infty}(s,\nuovoparametrotemporale)
=0, ~
\partial_s \kappa_{\gamma_\infty}(s,\nuovoparametrotemporale)
 \neq 0\Big\} = \emptyset.
$$
On the other hand, $\gamma_\infty$ evolves by curvature (see 
step 2), and therefore,
from the results of \cite{An:91}, if there exists 
$(s,\nuovoparametrotemporale) \in I_\infty \times\R$ such that  
$\kappa_{\gamma_\infty}(s,\nuovoparametrotemporale)
=0$ and 
$\partial_s \kappa_{\gamma_\infty}(s,\nuovoparametrotemporale)
 =0$,
then 
$\gamma_\infty(\cdot,\nuovoparametrotemporale)$ 
is linear, hence
$\gamma_\infty(\cdot,\cdot)$ 
is linear. 
Since this 
is in contradiction with 
\eqref{siconciau}, the proof of step 3 is concluded.

\medskip

{\it Step 4}. $I_\infty \neq \R$. 

Indeed, assume by contradiction that $I_\infty = \R$.
{}From step 3, reasoning as in \cite[pp. 512-513]{Al:91}
it follows that $\gamma_\infty$ 
is the so-called {\it grim reaper}. 
For the grim reaper
the function $Q_1^{\gamma_\infty} : \R \to (0,+\infty)$ defined on 
the right hand side of 
\eqref{defQ1} (with $[0,1]$  replaced by $I_\infty$) 
is identically zero.
On the other hand, from 
\eqref{eq:invresc} and arguing as in 
step 4 of the proof of Theorem \ref{th:singuno}  
we have that 
 $g_{\widehat \gamma_h} : [- \paraII_h^2 t_h, \paraII_h^2 (T- t_h)) \to (0,+\infty)$ 
is bounded from below by a  positive
constant uniformly with respect to $h \in \mathbb N$. 
Recall now that the sequence
$\{\widehat\gamma_h\}$ converges in $\C^2_{{\rm loc}}(I_\infty \times \R;\R^2)$
to $\gamma_\infty$
and that we have
(similarly to the inequality in \eqref{quant})
\begin{equation}\label{squic}
Q_1^{\gamma_\infty}(\nuovoparametrotemporale) \geq 
\limsup_{h \to +\infty} 
Q_1^{\widehat \gamma_h}(\nuovoparametrotemporale)\geq 
\limsup_{h \to +\infty} 
g_{\widehat \gamma_h}(\nuovoparametrotemporale).
\qquad \nuovoparametrotemporale\in\R.
\end{equation}
Then \eqref{squic} is 
in contradiction with  $Q_1^{\gamma_\infty}\equiv 0$, and the proof
of step 4 is concluded.

\medskip

Thanks to step 3 we can consider only two cases: either
$\kappa_{\gamma_\infty}(s,\nuovoparametrotemporale)<0$
for any $(s, \nuovoparametrotemporale) \in I_\infty \times \R$, 
or 
$\kappa_{\gamma_\infty}(s,\nuovoparametrotemporale)>0$
for any $(s, \nuovoparametrotemporale) \in I_\infty \times \R$.
Let us first assume 
\begin{equation}\label{durg}
\kappa_{\gamma_\infty}(s,\nuovoparametrotemporale)<0, \qquad
(s, \nuovoparametrotemporale) \in I_\infty \times \R.
\end{equation}
Recalling our conventions
(see Remark \ref{rem:conve}), inequality \eqref{durg} implies that 
$\gamma_\infty(\cdot,\nuovoparametrotemporale)$
is a convex curve. 

\medskip

{}From step 4 we have that either $a_\infty$ is finite or $b_\infty$
is finite. We assume that $a_\infty \in (-\infty,0]$, the case
$b_\infty \in [0,+\infty)$ being analogous. Therefore we have
$$
I_\infty= [a_\infty,+\infty).
$$
Observe that from \eqref{fixarcnew} we have 
\begin{equation}\label{dodiciottoa}
{\gamma_\infty}_2(a_\infty,
\nuovoparametrotemporale) = {\gamma_\infty}_2(a_\infty,0), 
\qquad \nuovoparametrotemporale \in \R.
\end{equation}
Recall also (see step 2) that 
\begin{equation}\label{dodiciottob}
\partial_s \gamma_\infty(a_\infty,
\nuovoparametrotemporale) = \left(\frac{1}{2}, \frac{\sqrt{3}}{2}\right),
\qquad \nuovoparametrotemporale \in \R.
\end{equation}
\medskip

{\it Step 5}. 
We have 
\begin{equation}\label{bonfi}
\int_{I_\infty} 
\kappa_{\gamma_\infty}(s,\nuovoparametrotemporale) ~ds \in [-\pi/3, 0),
\qquad \nuovoparametrotemporale \in \R.
\end{equation}
Indeed, if by contradiction 
there exists $\nuovoparametrotemporale \in \R$ such that 
the left hand side of \eqref{bonfi} 
is less than $-\pi/3$, then 
thanks to \eqref{durg} and the Neumann boundary condition \eqref{dodiciottob}, 
the curve $\gamma_\infty(\cdot,\nuovoparametrotemporale)$ has 
another intersection
 (different from $\gamma_\infty(a_\infty,0)$)
 with 
the horizontal axis $\ell$ passing from $\gamma_\infty(a_\infty,0)$. 
This
 implies $Q_2^{\gamma_\infty} \equiv 0$, 
where $Q_2^{\gamma_\infty}$ is defined as in \eqref{defQ1} (with $[0,1]$
replaced by $I_\infty$, and $\gamma_\infty^{{\rm sp}}$ 
is now the specular of $\gamma_\infty$ with respect to $\ell$). This 
leads to
a contradiction, as in step 4. 

\smallskip

In particular, the convex curve 
$\gamma_\infty(\cdot,\nuovoparametrotemporale)$ can be written as the 
graph of a 
strictly concave smooth function $y=
y(x,\nuovoparametrotemporale)$, where $(x,\nuovoparametrotemporale)\in 
[{\gamma_\infty}_1(a_\infty,\nuovoparametrotemporale),+\infty)\times\R$.

Let $\theta(x,\nuovoparametrotemporale):=\tan^{-1}(y_x(x,\nuovoparametrotemporale))\in (0,\pi/3]$ 
be the angle that the tangent vector to 
$\gamma_\infty(\cdot,\nuovoparametrotemporale)$ makes with 
the first 
coordinate axis. 

\medskip

{\it Step 6}. We have 
\begin{equation}\label{mare}
\partial_{\nuovoparametrotemporale}\kappa_{\gamma_\infty}(s,
\nuovoparametrotemporale)
\leq 0, \qquad 
(s,\nuovoparametrotemporale) \in I_\infty \times\R.
\end{equation}
Write for simplicity
\begin{equation}\label{lasido}
\kappa_{\gamma_\infty} = \kappa.
\end{equation}
Recalling that 
$\gamma_\infty$ evolves by curvature,
the evolution of $\kappa$ in the $(\theta,\nuovoparametrotemporale)$-coordinates 
reads as follows (see \cite{GH:86}):
\begin{equation}\label{eqkk}
\partial_{\nuovoparametrotemporale} \kappa
= \kappa^2 \kappa_{\theta\theta} 
+ \kappa^3.
\end{equation}
Let $\nuovoparametrotemporale_1\in\R$ and define 
$h:=\kappa
+2(\nuovoparametrotemporale-\nuovoparametrotemporale_1)
\partial_\nuovoparametrotemporale
\kappa
$. 
We have $h(\theta,\nuovoparametrotemporale_1)< 0$ for any $\theta \in 
(0,\pi/3]$, and
\begin{equation}\label{rompo}
h_{\nuovoparametrotemporale}  
= \kappa^2 h_{\theta\theta} + 
\left( \kappa^2 + 
\frac{
2
\partial_{\nuovoparametrotemporale}\kappa
}{\kappa}\right) h.
\end{equation}
Moreover, from $\partial_s = \kappa \partial_\theta$
and  \eqref{eqkappasdiffdue} we have that $h$ satisfies the 
 boundary condition
\begin{equation}\label{soffiare}
h_\theta
\left(\frac{\pi}{3},\nuovoparametrotemporale\right) 
= \frac{1}{\sqrt 3} ~h
\left(\frac{\pi}{3},\nuovoparametrotemporale\right), \qquad 
\nuovoparametrotemporale \in \R.
\end{equation}
We now observe that the 
remaining Dirichlet boundary condition for $h$ reads as
\begin{equation}\label{nonimme}
h(0,\nuovoparametrotemporale) =0, \qquad \nuovoparametrotemporale\in\R.
\end{equation}
Indeed, from \eqref{durg} and 
\eqref{bonfi} and the Lipschitz continuity of
$\kappa$ in $s$, which is uniform 
with respect to $\nuovoparametrotemporale$ (this follows from \eqref{siconciau} and
the interior regularity estimates in \cite{EcHu:91}),
we have 
\begin{equation}\label{ciaba}
\lim_{\theta \to 0^+} 
\kappa(\theta, \npt) = 0, \qquad \npt \in \R.
\end{equation} 
Using again \cite{EcHu:91} we deduce 
\begin{equation}\label{tte}
\lim_{\theta \to 0^+} 
\kappa_\theta(\theta, \npt) = 
\lim_{\theta \to 0^+} 
\kappa_{\theta\theta}(\theta, \npt) = 0, \qquad \npt \in \R.
\end{equation}
 Then 
\eqref{nonimme} follows from \eqref{ciaba} and \eqref{tte}. 
 
\smallskip

By \eqref{rompo}, \eqref{soffiare}
, \eqref{nonimme} and the maximum principle 
it then follows $h(\theta,\nuovoparametrotemporale)\leq 0$ 
for all $\theta \in (0,\pi/3]$ 
and $\nuovoparametrotemporale\ge\nuovoparametrotemporale_1$, hence
\[
\partial_{\nuovoparametrotemporale}
\kappa
 \leq - \frac{\kappa}{2(\nuovoparametrotemporale-\nuovoparametrotemporale_1)},
\qquad 
\nuovoparametrotemporale >
\nuovoparametrotemporale_1,
\]
which implies \eqref{mare}, by letting $\nuovoparametrotemporale_1\to -\infty$.

\medskip

{\it Step 7}. We have 
\begin{equation}\label{mareotto}
\partial_{\nuovoparametrotemporale}\kappa_{\gamma_\infty}
(s,\nuovoparametrotemporale)
= 0, \qquad 
(s,\nuovoparametrotemporale) \in I_\infty \times\R.
\end{equation}

Let us adopt the notation in \eqref{lasido}, 
and define
$Z(\nuovoparametrotemporale):=
\int_{0}^{\pi/3}\partial_{\nuovoparametrotemporale} 
(\log(-\kappa))\,d\theta$. 
Notice that $Z\ge 0$ since 
$\partial_{\nuovoparametrotemporale}\kappa \le 0$ by step 6 
and $\kappa <0$ by \eqref{durg}. Step 7 will be proved
if we show that 
\begin{equation}\label{Zzero}
Z \equiv 0.
\end{equation}
Following \cite[Section 8]{Al:91} we compute
\begin{equation}\label{chie}
\kappa_{\nuovoparametrotemporale\nuovoparametrotemporale}
= 
(\kappa^2 \kappa_{\theta\theta}+
\kappa^3)_{\nuovoparametrotemporale}
= 
\kappa^2(\kappa_{\theta\theta 
\nuovoparametrotemporale} + \kappa_t)+2\frac{(\kappa_\nuovoparametrotemporale
)^2}{\kappa}.
\end{equation}
Using \eqref{chie} and integrating by parts we get
\begin{equation}\label{devopassar}
\begin{aligned}
 Z'(\nuovoparametrotemporale) =&
\int_{0}^{\pi/3}\partial_{\nuovoparametrotemporale} 
\left(
\frac{\kappa_\nuovoparametrotemporale}{\kappa}\right)\,d\theta\,
\\
=& 
\int_{0}^{\pi/3}
\frac{\kappa_{\nuovoparametrotemporale\nuovoparametrotemporale}}{\kappa}
\,d\theta - 
\int_0^{\pi/3} \frac{\kappa_\nuovoparametrotemporale (\kappa^2 \kappa_{\theta
\theta} + \kappa^3)}{\kappa^2} d \theta
\\
=&
\int_{0}^{\pi/3}
\kappa
(\kappa_{\theta \theta \nuovoparametrotemporale}
+ \kappa_{\nuovoparametrotemporale}) + 2\frac{(\kappa_\nuovoparametrotemporale)^2}{\kappa^2}
\,d\theta - 
\int_0^{\pi/3} \frac{\kappa_\nuovoparametrotemporale (\kappa^2 \kappa_{\theta
\theta} + \kappa^3)}{\kappa^2} d \theta
\\ 
=& 
\int_{0}^{\pi/3}
\kappa
\kappa_{\theta \theta \nuovoparametrotemporale}
- 
\kappa_\nuovoparametrotemporale 
\kappa_{\theta
\theta} 
+ 2\frac{(\kappa_\nuovoparametrotemporale)^2}{\kappa^2}  
d \theta
\\
=&
\kappa(\pi/3,\nuovoparametrotemporale)
 \kappa_{\theta \nuovoparametrotemporale}
(\pi/3,\nuovoparametrotemporale)
- 
\kappa_\theta(\pi/3,\nuovoparametrotemporale)
 \kappa_{\nuovoparametrotemporale}(\pi/3,\nuovoparametrotemporale)
+ 
 2 \int_0^{\pi/3}\frac{(\kappa_\nuovoparametrotemporale)^2}{\kappa^2}  
d \theta.
\end{aligned}
\end{equation}
We now observe that from $\kappa_s = \kappa \kappa_\theta$
and from \eqref{behh} we have
$$
\kappa_\theta(\pi/3, \nuovoparametrotemporale) 
= \frac{\kappa(\pi/3, \nuovoparametrotemporale)}{\sqrt{3}},
 \qquad \nuovoparametrotemporale 
\in \R.
$$
Differentiating this 
relation 
with respect to $\nuovoparametrotemporale$ we obtain
\begin{equation}\label{pompiere}
\kappa(\pi/3,\nuovoparametrotemporale)
 \kappa_{\theta \nuovoparametrotemporale}
(\pi/3,\nuovoparametrotemporale)=
\kappa_\theta(\pi/3,\nuovoparametrotemporale)
 \kappa_{\nuovoparametrotemporale}(\pi/3,\nuovoparametrotemporale),
 \qquad \nuovoparametrotemporale 
\in \R.
\end{equation}
{}From \eqref{devopassar}, \eqref{pompiere} and 
the Schwarz's inequality we deduce
$$
Z'(\nuovoparametrotemporale) = 
 2 \int_0^{\pi/3}\frac{(\kappa_\nuovoparametrotemporale)^2}{\kappa^2}  
d \theta
=  2\int_{0}^{\pi/3}\left(
\partial_{\nuovoparametrotemporale} (\log(-\kappa))\right)^2\,d\theta
\ge \frac{6 Z^2(\nuovoparametrotemporale)}{\pi}.
$$
Assume now that $Z(\nuovoparametrotemporale_1)>0$ for some $\nuovoparametrotemporale_1\in\R$. It follows that $Z(\nuovoparametrotemporale)\ge Z(\nuovoparametrotemporale_1)>0$ 
for all $\nuovoparametrotemporale\ge\nuovoparametrotemporale_1$, 
which implies 
\[
Z(\nuovoparametrotemporale_1) \le \frac{1}{\frac{1}{Z(\nuovoparametrotemporale_2)}+\frac 6 \pi (\nuovoparametrotemporale_2-\nuovoparametrotemporale_1)} \le \frac{\pi}{6(\nuovoparametrotemporale_2-\nuovoparametrotemporale_1)}
\]
for all $\nuovoparametrotemporale_2\ge \nuovoparametrotemporale_1$.
Letting $\nuovoparametrotemporale_2\to +\infty$ we get $Z(\nuovoparametrotemporale_1) \leq 0$,
a contradiction. Hence \eqref{Zzero} follows, and the proof of step 7
is concluded.

\medskip 

{\it Step 8}. Assume now that
\begin{equation}\label{durgpos}
\kappa_{\gamma_\infty}(s,\nuovoparametrotemporale)>0, \qquad
(s, \nuovoparametrotemporale) \in I_\infty \times \R.
\end{equation}

Reasoning as in step 5 we have 
\begin{equation}\label{bonfipos}
\int_{I_\infty} 
\kappa_{\gamma_\infty}(s,\nuovoparametrotemporale) ~ds \in (0, 2\pi/3],
\qquad \nuovoparametrotemporale \in \R.
\end{equation}
Note that in this case the image of $\gamma_\infty(\cdot,\npt)$
is not necessarily a graph, but still the function
$\theta$ is well-defined, thanks to \eqref{durgpos},
and takes values in 
$[\pi/3, \pi)$. 
 Reasoning
as in steps 6 and 7, using the boundary conditions \eqref{soffiare} and 
$$
h(0,\npt) = \pi, \qquad \npt \in \R,
$$
and the choice
$Z(\npt) := \int_{\pi/3}^\pi \partial_\npt
(\log \kappa)~d\theta$, we deduce that
\eqref{mareotto} is still valid.

\medskip

{\it Step 9}. $\gamma_\infty$ is one of the 
two specific pieces of the grim reaper depicted in Fig. \ref{grim}.

{}From step 7 and \eqref{eqkk} we have $\partial_{\theta\theta}\kappa_{\gamma_\infty} 
+ \kappa_{\gamma_\infty} =0$.
By direct integration 
and using \eqref{dodiciottob}, it follows 
that $\gamma_\infty$ is a one-parameter
family of pieces of grim reapers
(the parameter being for instance the horizontal velocity of 
translation), see Fig.  \ref{grim}. As in step 5,
we have $Q_2^{\gamma_\infty} \equiv 0$, which gives a contradiction. 
This shows that $\gamma$ cannot develop type II singularities, and
concludes the proof of the theorem.
\end{proof}

\begin{figure} 
\begin{center}
\includegraphics[height=1.4cm]{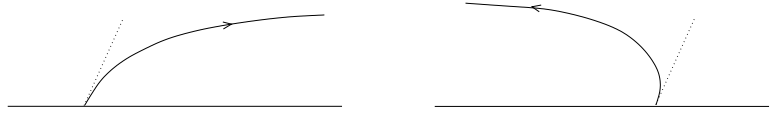}
\caption{{\small
Two pieces of the grim reaper, with the given $\pi/3$-Neumann boundary
condition.}}
\label{grim}
\end{center}
\end{figure}
%
\section{Examples}\label{secexa}
In the first example 
we show a graph-like initial datum 
$\inidat$  which develops a type II
singularity:
differently from Section \ref{secpostwo} (see \eqref{asspos}), in this case 
$\inidat_2$ changes sign. 

\subsection{Example 1}\label{secexa1}
For $x\in [0,1]$ let $\inidat(x):=(x,\overline f(x))$ 
where $\overline f$ is a smooth function the graph of which satisfies
the Neumann boundary conditions \eqref{neumangle} at $x=0$ and $x=1$, with 
the property that there exist $x_1, x_2 \in (0,1)$, 
$x_1<x_2$, such that $\overline f>0$ on $(0,x_1) \cup (x_2,1)$, 
and $\overline f<0$ on $(x_1,x_2)$ (see Fig. \ref{exa_typeII}). Set 
\begin{eqnarray*}
\int_0^{x_1}\overline f(x)\,dx =: \varepsilon >0, 
\qquad \int_{x_1}^{x_2} \overline f(x)\,dx =: -c <0.
\end{eqnarray*}
\begin{figure}
\begin{center}
\includegraphics[height=3.5cm]{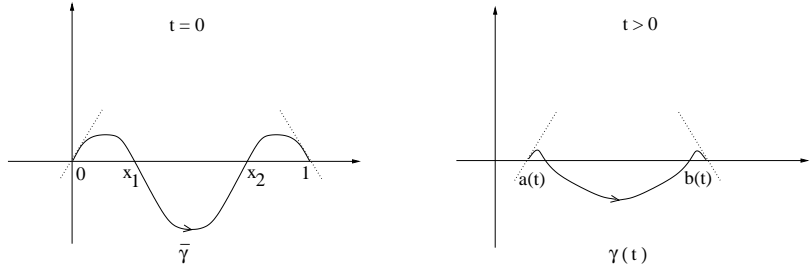}
\caption{{\small Example 1: the initial datum (left) 
and its evolution (right), which 
develops 
a type II singularity before the extinction.}}
\label{exa_typeII} 
\end{center}
\end{figure}
Then the image of $\gamma(t)$ can be written as the graph, over 
a smoothly variable interval $[a(t),b(t)]$, of a smooth function 
$f(\cdot,t):[a(t),b(t)]\to \R$, for $t\in [0,T)$, which solves
the problem
\begin{equation}\label{eqevolf}\left\{
\begin{array}{ll}
f_t &=\, \dfrac{f_{xx}}{1+(f_x)^2} \qquad \qquad {\rm in}~ (a(t),b(t)) \times (0,T),
\\
f(a(t),t)&=\,f(b(t),t)\,=\,0 
\quad \quad
t \in (0,T), 
\\
f_x(a(t),t) &= \,\sqrt 3 
\qquad \qquad \qquad \ \
t \in (0,T), 
\\
f_x(b(t),t) &=\, - \sqrt 3
\qquad \qquad \qquad 
t \in (0,T), 
\\
a(0) &= 0 
\\
b(0) &=\, 1
\\
f(\cdot,0) &= \,\overline f(\cdot) \qquad \qquad \qquad \  \ {\rm in}~ (0,1),
\end{array}\right.
\end{equation}
where,
for notational simplicity, we still
denote by $x$ the first variable in $\R^2$. 

By the maximum principle for $f_x$ (see 
\cite{SS:08}) the functions $f(\cdot,t)$ are Lipschitz 
continuous, with a Lipschitz constant
which is  uniform with respect to $t\in [0,T)$. 
By the smoothness of the flow, 
there exist $t_s \in (0,T]$ 
and two continuous functions $x_1, x_2 : [0,t_s) \to \R$,
with $a(t)<x_1(t)<x_2(t)<1$ for any $t \in [0,t_s)$, such that 
$x_i(0)=x_i$, $i=1,2$, $f(\cdot,t)>0$ on $(a(t),x_1(t))
\cup (x_2(t),1)$,
and $f(\cdot,t)<0$ on $(x_1(t),x_2(t))$.
Define, for any $t \in (0,t_s)$, the nonnegative functions
\[
V^+(t) := \int_{a(t)}^{x_1(t)} f(x,t)\,dx, 
\qquad V^-(t) := -\int_{x_1(t)}^{x_2(t)} f(x,t)\,dx. 
\]
By a direct computation, we get
\[
\frac{d}{dt} V^+(t) \le -\frac\pi 3,
\qquad \frac{d}{dt} V^-(t) \ge -\pi ,
\]
so that 
\begin{equation}\label{arte}
V^+(t)\le \varepsilon - \frac\pi 3 t,
\qquad V^-(t)\ge c-\pi t, \qquad t \in (0,t_s).
\end{equation}
Observe that if there exists $\overline t \in (0,t_s]$ such that 
$V^+>0$ in $[0,\overline t)$,  
$V^+(\overline t)=0$ (hence $a(\overline t) = x_1(\overline
t)$) and $V^->0$ in $[0,\overline t]$, then 
$\overline t$ is a singularity time due to the 
boundary conditions (and $\overline t$ is not the extinction time).
Hence, from \eqref{arte} it follows that 
if $\varepsilon$ is small enough, i.e.
$c-3\varepsilon >0$, a singularity occurs {\it before} the extinction 
of the evolution.
It follows that $t_s = T\le 3\varepsilon/\pi$.

Reasoning as in Theorem \ref{th:singuno}, 
we can exclude that $\gamma(t)$ develops 
type I singularities at $t = T$: indeed, developing
a type I 
singularity at $t=T$ would imply a nontrivial homotetic solution obtained
as a blow up,
which (thanks to the boundary conditions) is unique,
and 
would correspond to the extinction at $t=T$, 
which contradicts $\liminf_{t \to T^-} V^-(t)>0$. 
It  follows 
that $\gamma(t)$ develops a type II singularity at $t=T$. Arguing as in 
the proof of Theorem \ref{teosingII}, 
a suitable rescaled and translated version of $\gamma(t)$ 
converges either to a grim reaper 
or to a piece of the grim reaper  with a boundary point. 
In fact, we can rule out the first possibility, since the grim reaper 
cannot be written as the graph of a Lipschitz function.
We conclude that if $\epsilon < c/3$ 
a type II singularity
(the blow-up of which is as in Fig. \ref{grim})
 must occur before the extinction time.

\medskip
In the next example we show a singularity
due to collision of the boundary points, happening
before the extinction time.

\subsection{Example 2}\label{subexa2}
Let us consider an evolution similar to \eqref{eqevol},
where we 
substitute the boundary conditions on $\tangvers(0,t)$ and $\tangvers(1,t)$ with 
\begin{equation}\label{lotte}
\tangvers(0,t) = \left( -\frac 1 2, \frac{\sqrt 3}{2}\right)
\qquad 
\tangvers(1,t) = \left( -\frac 1 2, -\frac{\sqrt 3}{2}\right),
\end{equation}
so that the angle between 
$e_1$ and 
$\tangvers(t)$ equals 
$2\pi/3$ at 
$ \gamma(0,t) = (\gamma_1(0,t),0)$, and 
equals $ -2 
\pi/3$ at 
$\gamma(1,t) = (\gamma_1(1,t),0)$.

We still assume that $\inidat$ is smooth and embedded, 
with $\inidat_2> 0$ in $(0,1)$ as in Sections \ref{secpos},
\ref{secpostwo} (see Fig. \ref{centoventi}).
\begin{figure} 
\begin{center}
\includegraphics[height=2.0cm]{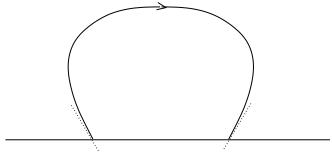}
\caption{{\small 
The initial datum in Example 2.
}}
\label{centoventi}
\end{center}
\end{figure}
At the singular time $t=T$ either \eqref{ksing} holds 
or the curvature stays bounded but there is a collision of the 
boundary points, i.e.
\begin{equation}\label{eqlength}
\liminf_{t\to T^-}|\gamma_1(1,t)-\gamma_1(0,t)| = 0.
\end{equation}
Notice that this is impossible for the solutions of \eqref{eqevol}, due to the boundary conditions.

Since Theorem \ref{th:singuno} applies also to this situation, 
we can exclude the formation of type I singularities before the extinction time.
Moreover, since $\inidat$ is embedded and $\inidat_2$ is positive in $(0,1)$, 
we can also exclude type II 
singularities, reasoning exactly as in Section \ref{secpostwo}.

Assume now that 
$T$ is the extinction time of the evolution, and that the evolution 
develops a  
type I singularity at $t=T$. By the analysis in 
Section \ref{secpos}, it follows 
that the evolution converges, after rescaling, to a homothetic solution.
However there are no such solutions compatible with the boundary conditions 
\eqref{lotte}, see \cite{CG:07}, \cite{H:07}. Hence
$T$ is not the extinction time of the evolution 
and \eqref{eqlength} necessarily holds. 
A collision of the boundary points occurs as $t \to T^-$,
while the curvature remains bounded. 


\end{document}